\documentclass[11pt, a4paper,leqno]{amsart}
\usepackage{amsmath,amsthm,amscd,amssymb,amsfonts, amsbsy}
\usepackage{latexsym}
\usepackage{exscale}
\usepackage{txfonts}

\usepackage[colorlinks,citecolor=red,pagebackref,hypertexnames=false]{hyperref}

\usepackage{pgf}
\usepackage{color}

\calclayout
\allowdisplaybreaks

\theoremstyle{plain}
\newtheorem{theorem}[equation]{Theorem}
\newtheorem{lemma}[equation]{Lemma}
\newtheorem{corollary}[equation]{Corollary}
\newtheorem{proposition}[equation]{Proposition}

\newtheorem{claim}[equation]{Claim}

\theoremstyle{definition}
\newtheorem{definition}[equation]{Definition}

\theoremstyle{remark}
\newtheorem{remark}[equation]{Remark}

\numberwithin{equation}{section}

\newcommand{\eps}{\varepsilon}

\newcommand{\dist}{\operatorname{dist}}

\newcommand{\td}{\Delta^{\star}}
\newcommand{\wtd}{\widetilde{\Delta}^\star_X}

\newcommand{\ds}{\delta_{\star}}

\newcommand{\ttom}{\widehat{\Omega}}
\newcommand{\tthm}{\widehat{\omega}}

\newcommand{\sss}{\sigma_\star}

\newcommand{\dv}{\operatorname{div}}

\newcommand{\re}{\mathbb{R}}
\newcommand{\rn}{\mathbb{R}^n}

\newcommand{\reu}{\mathbb{R}^{n+1}_+}
\newcommand{\ree}{\mathbb{R}^{n+1}}

\newcommand{\A}{\Lambda}
\newcommand{\dd}{\mathbb{D}}
\newcommand{\ppn}{\mathbb{P}_N}
\newcommand{\tppn}{\widetilde{\mathbb{P}}_N}
\newcommand{\ppsn}{\mathbb{P}_N^\star}
\newcommand{\tppsn}{\widetilde{\mathbb{P}}^\star_N}

\newcommand{\C}{\mathcal{C}}

\newcommand{\om}{\Omega}
\newcommand{\ot}{\Omega_0}

\newcommand{\F}{\mathcal{F}}

\newcommand{\M}{\mathcal{M}}

\newcommand{\W}{\mathcal{W}}
\newcommand{\tU}{\widetilde{U}}
\newcommand{\tQ}{\widetilde{Q}}
\newcommand{\tQs}{\widetilde{Q}_\star}
\newcommand{\Qs}{Q_\star}
\newcommand{\tB}{\widetilde{B}}

\newcommand{\B}{\mathcal{B}}

\newcommand{\oo}{\mathcal{O}}

\newcommand{\mut}{\mathfrak{m}}
\newcommand{\mutt}{\widetilde{\mathfrak{m}}}

\newcommand{\pom}{\partial\Omega}

\newcommand{\hm}{\omega}

\newcommand{\bqo}{B_Q^{\rm big}}
\newcommand{\dqo}{\Delta_Q^{\rm big}}

\renewcommand{\P}{\mathcal{P}}

\newcommand{\oT}{\widetilde{\Omega}}

\renewcommand{\emptyset}{\mbox{\textup{\O}}}

\DeclareMathOperator{\diam}{diam}

\DeclareMathOperator{\interior}{int}

\begin{document}
\allowdisplaybreaks

\title[Uniform Rectifiability and harmonic measure]{Uniform Rectifiability
and harmonic measure IV:  Ahlfors regularity plus Poisson kernels in $L^p$
implies uniform rectifiability.}

\author{Steve Hofmann}

\address{Steve Hofmann
\\
Department of Mathematics
\\
University of Missouri
\\
Columbia, MO 65211, USA} \email{hofmanns@missouri.edu}

\author{Jos\'e Mar{\'\i}a Martell}

\address{Jos\'e Mar{\'\i}a Martell\\
Instituto de Ciencias Matem\'aticas CSIC-UAM-UC3M-UCM\\
Consejo Superior de Investigaciones Cient{\'\i}ficas\\
C/ Nicol\'as Cabrera, 13-15\\
E-28049 Madrid, Spain} \email{chema.martell@icmat.es}



\thanks{The first author was supported by NSF grant DMS-1361701.
The second author was supported by ICMAT Severo Ochoa project SEV-2011-0087. He also acknowledges that the research leading to these results has received funding from the European Research Council under the European Union's Seventh Framework Programme (FP7/2007-2013)/ ERC agreement no. 615112 HAPDEGMT.}

\date{\today}

\begin{abstract}
Let $E\subset \ree$, $n\ge 2$, be an Ahlfors-David regular set of dimension $n$.
We show that the weak-$A_\infty$ property of harmonic measure, for the open set
$\Omega:= \ree\setminus E$, implies uniform rectifiability of $E$.
\end{abstract}


\bigskip
\bigskip


\maketitle

{\bf This is a preliminary version of our work on this topic, 
as presented by the first author at the Workshop on Harmonic Analysis and PDE
held at ICMAT in Madrid, in January 2015.  The final published version will be jointly 
authored with K. Nystr\"om and P. Le,
and in addition to the present results, will treat also the analogous theory for the $p$-Laplacian.}

\tableofcontents

\section{Introduction}\label{s1}


In this note, we present a quantitative, scale invariant result of free boundary type.
Somewhat more precisely, let $\om\subset \ree $ be an open set (not necessarily connected)
satisfying an interior Corkscrew condition,
whose boundary is $n$-dimensional Ahlfors-David regular (ADR).   Given these background hypotheses,
we show that if $\hm$, the
harmonic measure for $\om$,
is absolutely continuous with respect to surface measure
$\sigma$, and if the Poisson kernel $k=d\hm/d\sigma$ verifies an appropriate scale invariant higher 
integrability estimate (in particular, if $\hm$ belongs to weak-$A_\infty$ with respect to $\sigma$),
then $\pom$ is uniformly rectifiable, in the sense of \cite{DS1,DS2}.  
Here $\sigma:= H^n|_{\pom}$
is, as usual, the restriction to $\pom$ of $n$-dimensional Hausdorff measure (the 
other notation and terminology used here will be defined in the sequel).   
In particular, our background hypotheses hold in the case that $\om:=\ree\setminus E$ is the
complement of an ADR (hence closed) set of co-dimension 1:  in that case, it is well known that
the Corkscrew condition is verified automatically in $\om$,  i.e., in every ball $B=B(x,r)$
centered on $\pom$, there is some component of $\om\cap B$ that contains a point $Y$ with
$\dist(Y,\pom)\approx r$.  

In previous work with I. Uriarte-Tuero \cite{HMU}, the authors had proved such a result under the additional hypothesis
that $\om$ is a connected domain, satisfying an interior Harnack Chain condition.  In hindsight, under that extra
assumption, one obtains the stronger conclusion that in fact, $\om_{ext}:=\ree\setminus \overline{\om}$ also satisfies a
Corkscrew condition, and hence that $\om$ is an NTA domain in the sense of \cite{JK}; see \cite{AHMNT} for the details.
The new advance in the present paper, then, is the removal of any connectivity hypothesis; in particular, we avoid
the Harnack Chain condition.

Before discussing further historical background, let us now state our main result.  To this end, 
given an open set $\om\subset \ree$, and a Euclidean ball $B=B(x,r)\subset \ree$, 
centered on $\pom$, we let $\Delta =\Delta(x,r):= B\cap\pom$ 
denote the corresponding surface ball, and for a constant $C>0$, we set $C\Delta:=\Delta(x,Cr)$.  
For $X\in \om$, let $\hm^X$ be
harmonic measure for $\om$, with pole at $X$.
As mentioned above,
all other terminology and notation will be defined below.

\begin{theorem}\label{t1} Let $\om\subset \ree$, $n\ge 2$, be an 
open set, whose boundary is Ahlfors-David regular of dimension $n$.
Suppose that there is a constant $C_0\ge 1$, and an exponent $p>1$, such that for every ball 
surface ball $\Delta=\Delta(x,r)$,
with $x\in \pom$ and $0<r<\diam \pom$, there exists $Y_{\Delta}\in B(x,C_0\,r)$ 
with $\dist(Y_\Delta,\pom)\ge C_0^{-1}\,r$, satisfying
\begin{list}{$(\theenumi)$}{\usecounter{enumi}\leftmargin=.8cm
\labelwidth=.8cm\itemsep=0.2cm\topsep=.1cm
\renewcommand{\theenumi}{\alph{enumi}}}

\item \textbf{Bourgain's estimate:} $\hm^{Y_\Delta}(\Delta)\ge C_0^{-1}$.

\item \textbf{Scale-invariant higher integrability:} $\hm^{Y_\Delta}\ll \sigma$ in $C_1\Delta$ and $k^{Y_\Delta}=d\hm^{Y_\Delta}/d\sigma$ satisfies
\begin{equation}\label{eqn:main-SI}
\int_{C_1\Delta} k^{Y_\Delta}(y)^p\,d\sigma(y)
\le
C_0\,\sigma(C_1\Delta)^{1-p},
\end{equation}
\end{list}
where $C_1$ is a large enough constant depending only on $n$ and the ADR constant of  $\pom$.
Then $\pom$ is uniformly rectifiable and moreover the ``UR character'' (see Definition \ref{defurchar}) depends only on $n$, the ADR constants, $p$ and $C_0$.
\end{theorem}

\noindent{\it Remark}.    As mentioned above, the background hypotheses hold in the special
case that $\om:=\ree\setminus E$
is the complement of an $n$-dimensional ADR set $E$, and in that setting, condition (a) is automatically
verified.  Indeed, 
by a result of Bourgain \cite{B} (see Lemma \ref{Bourgainhm} below), 
the Ahlfors-David regularity of the boundary
implies that there is always a point $Y_\Delta$ as above, and a sufficiently large
$C_0$, such that estimate $(a)$ holds;  
in fact, $\hm^Y$ satisfies (a) for every
$Y\in \om\cap B(x, c_1r)$, for  $c_1$ small enough depending only on 
dimension and the ADR constants, and as we have noted, the ADR property ensures that some such $Y$ satisfies
$\dist(Y,\pom)\approx r$.  Thus,  the theorem will hold in this setting, if (b) holds for this  $Y$.

The observations in the preceding remark will allow us to deduce,
as an easy corollary, the following variant of Theorem \ref{t1}
(we shall give the short proof of the corollary in Section \ref{s-cor}). 

\begin{corollary}\label{c1}
Let $E\subset \ree$, $n\ge 2$, be an Ahlfors-David regular set of dimension $n$, and let $\om:=\setminus E$.
Suppose that for every ball $B(x,r)$, $x\in E$, $0<r<\diam E$, and for all $Y\in \Omega\setminus B(x,2r)$,
harmonic measure $\hm^Y\in$ weak-$A_\infty(\Delta(x,r))$, that is, there is 
a constant $C_0\ge 1$ and an exponent $p>1$, each of which is
uniform with respect to $x,r$ and $Y$, 
such that $\hm^{Y}\ll \sigma$ in $\Delta(x,r)$, and $k^{Y}=d\hm^{Y}/d\sigma$ satisfies
\begin{equation}\label{eqn:main-weak-RHP}
\left(\fint_{\Delta'} k^{Y}(z)^p\,d\sigma(z)\right)^{\frac1p}
\le
C_0\,\fint_{2\,\Delta'} k^{Y}(z)\,d\sigma(z),
\end{equation}
for every $\Delta'=B'\cap E$ with $2\,B'\subset B(x,r)$.
Then $E$ is uniformly rectifiable and moreover the ``UR character'' (see Definition \ref{defurchar}) depends only on $n$, the ADR contant of $E$, $p$ and $C_0$.
\end{corollary}

Combining Theorem \ref{t1} with the results in \cite {BH},
we obtain as an immediate consequence a ``big pieces"
characterization of uniformly rectifiable sets of co-dimension 1, in terms of
harmonic measure.   Here and in the sequel, given an ADR set $E$,  $Q$ will denote a ``dyadic cube"
on $E$ in the sense of \cite{DS1,DS2} and \cite{Ch}, and $\dd(E)$ will denote the collection
of all such cubes;  see Lemma \ref{lemmaCh} below.

\begin{theorem}\label{t2} Let $E\subset \ree$ be an $n$-dimensional ADR set.
Let $\om := \ree\setminus E$. 
Then $E$ is uniformly rectifiable if and only if it
has ``big pieces of good harmonic measure estimates" 
in the  following sense: 
for each $Q \in \dd (E)$ there exists an open set $\oT=\oT_Q$ 
with the following properties, with uniform control of the various implicit
constants:
\begin{itemize}
\item $\partial\oT$ is ADR;
\item  the interior Corkscrew condition holds in $\oT$; 
\item $\partial\oT$ has a ``big pieces" overlap with $\pom$, in the sense that
\begin{equation}\label{bp}\sigma(Q\cap \partial\oT) \gtrsim \sigma(Q)\,;
\end{equation}
\item for each surface ball
$\Delta = \Delta(x,r) := B(x,r) \cap \partial\oT$, with $x \in \partial\oT$ and $r \in (0, \diam(\oT))$, 
there is an interior corkscrew point $X_\Delta\in \ot$, such that 
$\hm^{X_\Delta}:=\hm^{X_\Delta}_{\widetilde{\om}}$, the harmonic measure for $\oT$ with pole at $X_\Delta$, 
satisfies
$\hm^{X_\Delta}(\Delta) \gtrsim 1$, and
belongs to 
weak-$A_\infty(\Delta)$.
\end{itemize}
\end{theorem}
The ``only if" direction is proved in \cite{BH}, and in fact the open sets $\oT$ constructed there even 
satisfy a 2-sided Corkscrew condition, and moreover, $\oT\subset \om$,
with $\diam(\oT) \approx \diam(Q)$.    
To obtain the converse direction, we simply observe that by Theorem
\ref{t1}, the subdomains $\oT$ have uniformly rectifiable boundaries, with uniform control of the ``UR" character
of each $\partial\oT$, and thus $E$ is uniformly rectifiable, by \cite{DS2}.

Let us now discuss some related earlier results.  Our approach in the present paper
owes a great deal to prior work of Lewis and Vogel \cite{LV}, who proved a version of Theorem
\ref{t1} under the stronger hypothesis that $\hm$ itself is an Ahlfors-David regular measure,
and thus the Poisson kernel is a bounded, accretive function, i.e., $k\approx 1$.  With this assumption, 
they were able to show that $\pom$ satisfies the so-called ``Weak Exterior Convexity" (WEC) condition, which characterizes
uniform rectifiability  \cite{DS2}.    To weaken the hypotheses on $\hm$, as we have done here, requires
two further considerations.  The first is quite natural in this context:  a stopping time argument, in the spirit of the
proofs of the Kato square root conjecture \cite{HMc}, \cite{HLMc}, \cite{AHLMcT} (and of local $Tb$ theorems
\cite{Ch}, \cite{AHMTT}, \cite{H}), by means of which we extract ample dyadic sawtooth regimes on which averages
of harmonic measure are bounded and accretive (see Lemma \ref{l4.4} below).  This will allow us to
use the arguments of \cite{LV} within these good sawtooth regions.
The second new consideration
is necessitated by the fact that in our setting, the doubling property may fail for harmonic measure.  In the absence of doubling, 
we are unable to
obtain the  WEC condition directly.  Nonetheless, we shall be able to follow very closely the arguments of \cite{LV}
up to a point, to obtain a condition on $\pom$ that we have called the ``Weak Half Space Approximation" 
(WHSA) property.  Indeed, extracting the essence of the \cite{LV} argument, while dispensing with the doubling 
property, one realizes that the WHSA is precisely what one obtains.  To fix ideas, and for the sake of
self-containment, we shall summarize this fact, and present the argument of \cite{LV} here as Lemma \ref{LVlemma}.
Of course, in the proof of Lemma \ref{LVlemma},
we shall follow \cite{LV} quite closely.  Finally then, having obtained
that $\pom$ satisfies the WHSA property, we are then left with showing that WHSA implies uniform rectifiability:
\begin{proposition}\label{prop2.20}
An $n$-dimensional ADR set $E\subset \ree$ is uniformly rectifiable if and only if it satisfies the
WHSA property.
\end{proposition}
We shall give the definition of WHSA in Section \ref{s2}, 
and  the proof of the proposition in Section \ref{s7}.  While the WHSA condition, 
per se, is new, even in this last step we shall make use of a modified version of part of the argument in \cite{LV}.

 We note that in \cite{LV}, the authors treated also the case that the ``$p$-harmonic measure" (i.e., the
Riesz measure associated to a non-negative $p$-harmonic function vanishing on a surface ball) was ADR,
for all $1<p<\infty$.   Of course, the case $p=2$ corresponds to the classical Laplacian.
In a forthcoming joint paper with K. Nystr\"om and P. Le, we plan to extend the 
results of the present paper to the case of the $p$-Laplacian, with $1<p<\infty$.

To provide some additional context, we mention that out results here
may be viewed as a ``large constant'' analogue of a result of
Kenig and Toro \cite{KT3}, which states that in the presence of a Reifenberg flatness
condition and Ahlfors-David regularity,  $\log k \in VMO$ implies that the unit normal $\nu$ to the boundary  belongs to
$VMO$, where $k$ is the Poisson kernel with pole at some fixed point. 
Moreover, under the same background hypotheses,
the condition that $\nu \in VMO$ is equivalent to
a uniform rectifiability (UR) condition with vanishing trace,
thus $\log k \in VMO \implies vanishing \,\, UR,$ given sufficient Reifenberg flatness. On the other hand, our
large constant version ``almost'' says  ``$\,\log k \in BMO\implies UR\,$''.
Indeed,  it is well known that the $A_\infty$ condition
(i.e.,  weak-$A_\infty$ plus the doubling property) implies that $\log k \in BMO$, while
if $\log k \in BMO$ with small norm, then $k\in A_\infty$.  We further note that, in turn, the results of \cite{KT3}
may be viewed as an ``endpoint" version of the free boundary results of \cite{AC} and \cite{Je},
which say, again in the presence of Reifenberg flatness, that H\"older continuity of 
$\log k$ implies that of the unit normal $\nu$ (and indeed, that $\pom$ is of class $C^{1,\alpha}$ for some $\alpha>0$).

The paper is organized as follows.  In Section \ref{s2}, we state several definitions and basic lemmas.
In Section \ref{s4}, we begin the proof of Theorem \ref{t1} with some preliminary arguments, and in Section \ref{s6},
we complete the proof, modulo Proposition \ref{prop2.20}, following the arguments of \cite{LV}.  In Section \ref{s7},
we give the proof of Proposition \ref{prop2.20}, i.e., the proof of the fact that the WHSA condition 
implies uniform rectifiability.  Finally, in Section \ref{s-cor},  we give the (very short) proof of Corollary \ref{c1}.




\section{Preliminaries}\label{s2}

\begin{definition}\label{defadr} ({\bf  ADR})  (aka {\it Ahlfors-David regular}).
We say that a  set $E \subset \ree$, of Hausdorff dimension $n$, is ADR
if it is closed, and if there is some uniform constant $C$ such that
\begin{equation} \label{eq1.ADR}
\frac1C\, r^n \leq \sigma\big(\Delta(x,r)\big)
\leq C\, r^n,\,\,\,\forall r\in(0,\diam (E)),x \in E,
\end{equation}
where $\diam(E)$ may be infinite.
Here, $\Delta(x,r):= E\cap B(x,r)$ is the ``surface ball" of radius $r$,
and $\sigma:= H^n|_E$ 
is the ``surface measure" on $E$, where $H^n$ denotes $n$-dimensional
Hausdorff measure.
\end{definition}

\begin{definition}\label{defur} ({\bf UR}) (aka {\it uniformly rectifiable}).
An $n$-dimensional ADR (hence closed) set $E\subset \ree$
is UR if and only if it contains ``Big Pieces of
Lipschitz Images" of $\rn$ (``BPLI").   This means that there are positive constants $\theta$ and
$M_0$, such that for each
$x\in E$ and each $r\in (0,\diam (E))$, there is a
Lipschitz mapping $\rho= \rho_{x,r}: \rn\to \ree$, with Lipschitz constant
no larger than $M_0$,
such that
$$
H^n\Big(E\cap B(x,r)\cap  \rho\left(\{z\in\rn:|z|<r\}\right)\Big)\,\geq\,\theta\, r^n\,.
$$
\end{definition}

We recall that $n$-dimensional rectifiable sets are characterized by the
property that they can be
covered, up to a set of
$H^n$ measure 0, by a countable union of Lipschitz images of $\rn$;
we observe that BPLI  is a quantitative version
of this fact.

We remark
that, at least among the class of ADR sets, the UR sets
are precisely those for which all ``sufficiently nice" singular integrals
are $L^2$-bounded  \cite{DS1}.    In fact, for $n$-dimensional ADR sets
in $\ree$, the $L^2$ boundedness of certain special singular integral operators
(the ``Riesz Transforms"), suffices to characterize uniform rectifiability (see \cite{MMV} for the case $n=1$, and
\cite{NToV} in general).
We further remark that
there exist sets that are ADR (and that even form the boundary of a domain satisfying
interior Corkscrew and Harnack Chain conditions),
but that are totally non-rectifiable (e.g., see the construction of Garnett's ``4-corners Cantor set"
in \cite[Chapter1]{DS2}).  Finally, we mention that there are numerous other characterizations of UR sets
(many of which remain valid in higher co-dimensions); see \cite{DS1,DS2}, and in particular
Theorem \ref{t2.7} below.  In this paper, we shall also present a new characterization of UR sets of co-dimension 1
(see Proposition \ref{prop2.20} below),
which will be very useful in the proof of Theorem \ref{t1}.

\begin{definition}\label{defurchar} ({\bf ``UR character"}).   Given a UR set $E\subset \ree$, its ``UR character"
is just the pair of constants $(\theta,M_0)$ involved in the definition of uniform rectifiability,
along with the ADR constant; or equivalently,
the quantitative bounds involved in any particular characterization of uniform rectifiability.
\end{definition}

\begin{lemma}\label{lemmaCh}({\bf Existence and properties of the ``dyadic grid''})
\cite{DS1,DS2}, \cite{Ch}.
Suppose that $E\subset \ree$ is closed $n$-dimensional ADR set.  Then there exist
constants $ a_0>0,\, \gamma>0$ and $C_*<\infty$, depending only on dimension and the
ADR constants, such that for each $k \in \mathbb{Z},$
there is a collection of Borel sets (``cubes'')
$$
\mathbb{D}_k:=\{Q_{j}^k\subset E: j\in \mathfrak{I}_k\},$$ where
$\mathfrak{I}_k$ denotes some (possibly finite) index set depending on $k$, satisfying

\begin{list}{$(\theenumi)$}{\usecounter{enumi}\leftmargin=.8cm
\labelwidth=.8cm\itemsep=0.2cm\topsep=.1cm
\renewcommand{\theenumi}{\roman{enumi}}}

\item $E=\cup_{j}Q_{j}^k\,\,$ for each
$k\in{\mathbb Z}$

\item If $m\geq k$ then either $Q_{i}^{m}\subset Q_{j}^{k}$ or
$Q_{i}^{m}\cap Q_{j}^{k}=\emptyset$.

\item For each $(j,k)$ and each $m<k$, there is a unique
$i$ such that $Q_{j}^k\subset Q_{i}^m$.


\item $\diam\big(Q_{j}^k\big)\leq C_* 2^{-k}$.

\item Each $Q_{j}^k$ contains some ``surface ball'' $\Delta \big(x^k_{j},a_02^{-k}\big):=
B\big(x^k_{j},a_02^{-k}\big)\cap E$.

\item $H^n\big(\big\{x\in Q^k_j:{\rm dist}(x,E\setminus Q^k_j)\leq \varrho \,2^{-k}\big\}\big)\leq
C_*\,\varrho^\gamma\,H^n\big(Q^k_j\big),$ for all $k,j$ and for all $\varrho\in (0,a_0)$.
\end{list}
\end{lemma}

Let us make a few remarks are concerning this lemma, and discuss some related notation and terminology.

\begin{list}{$\bullet$}{\leftmargin=0.4cm  \itemsep=0.2cm}

\item In the setting of a general space of homogeneous type, this lemma has been proved by Christ
\cite{Ch}, with the
dyadic parameter $1/2$ replaced by some constant $\delta \in (0,1)$.
In fact, one may always take $\delta = 1/2$ (cf.  \cite[Proof of Proposition 2.12]{HMMM}).
In the presence of the Ahlfors-David
property (\ref{eq1.ADR}), the result already appears in \cite{DS1,DS2}.

\item  For our purposes, we may ignore those
$k\in \mathbb{Z}$ such that $2^{-k} \gtrsim {\rm diam}(E)$, in the case that the latter is finite.

\item  We shall denote by  $\mathbb{D}=\mathbb{D}(E)$ the collection of all relevant
$Q^k_j$, i.e., $$\mathbb{D} := \cup_{k} \mathbb{D}_k,$$
where, if $\diam (E)$ is finite, the union runs
over those $k$ such that $2^{-k} \lesssim  {\rm diam}(E)$.

\item Properties $(iv)$ and $(v)$ imply that for each cube $Q\in\mathbb{D}_k$,
there is a point $x_Q\in E$, a Euclidean ball $B(x_Q,r)$ and a surface ball
$\Delta(x_Q,r):= B(x_Q,r)\cap E$ such that
$r\approx 2^{-k} \approx {\rm diam}(Q)$
and \begin{equation}\label{cube-ball}
\Delta(x_Q,r)\subset Q \subset \Delta(x_Q,Cr),\end{equation}
for some uniform constant $C$.
We shall denote this ball and surface ball by
\begin{equation}\label{cube-ball2}
B_Q:= B(x_Q,r) \,,\qquad\Delta_Q:= \Delta(x_Q,r),\end{equation}
and we shall refer to the point $x_Q$ as the ``center'' of $Q$.

\item Given a dyadic cube $Q\in\dd$, we define its ``$\kappa$-dilate"  by
\begin{equation}\label{dilatecube}
\kappa Q:= E\cap B\left(x_Q,\kappa \diam(Q)\right).
\end{equation}

\item For a dyadic cube $Q\in \mathbb{D}_k$, we shall
set $\ell(Q) = 2^{-k}$, and we shall refer to this quantity as the ``length''
of $Q$.  Clearly, $\ell(Q)\approx \diam(Q).$

\item For a dyadic cube $Q \in \mathbb{D}$, we let $k(Q)$ denote the ``dyadic generation''
to which $Q$ belongs, i.e., we set  $k = k(Q)$ if
$Q\in \mathbb{D}_k$; thus, $\ell(Q) =2^{-k(Q)}$.

\end{list}

\begin{definition} ({\bf ``$\eps$-local BAUP"})\label{def2.4} Given $\eps>0$,
we shall say that $Q\in\dd(E)$ satisfies the
$\eps$-{\it local BAUP} condition if there is a family $\mathcal{P}$ of hyperplanes (depending on $Q$)
such that every point in $10Q$ is at a distance at most $\eps \ell(Q)$ from $\cup_{P\in\mathcal{P}}P$, and
every point in $\left(\cup_{P\in\mathcal{P}}P\right) \cap B(x_Q, 10\diam(Q))$ is at a distance at most $\eps\ell(Q)$
from $E$.

\end{definition}

\begin{definition}\label{def2.5} ({\bf BAUP}).
We shall say that an $n$-dimensional ADR set $E\subset \ree$
satisfies the condition of {\it Bilateral Approximation by
Unions of Planes}  (``BAUP"),
if for some $\eps_0>0$, and for every positive
$\eps<\eps_0$, there is a constant $C_0=C_0(\eps)$ such that the set $\B$ of bad cubes in $\dd(E)$, for which the
$\eps$-local
BAUP condition 
fails, satisfies the packing condition
\begin{equation}\label{eq2.pack}
\sum_{Q'\subset Q,\, Q'\in\B} \sigma(Q')\,\leq \,C_0 \,\sigma(Q)\,,\qquad \forall\, Q\in \dd(E)\,.
\end{equation}
\end{definition}

For future reference, we recall the following result of David and Semmes \cite{DS2}.

\begin{theorem}[{\cite[Theorem I.2.18, p. 36]{DS2}}]\label{t2.7}
Let $E\subset \ree$ be an $n$-dimensional ADR set. Then, $E$ is uniformly rectifiable if and only if it satisfies
BAUP.
\end{theorem}

We remark that the definition of BAUP in \cite{DS2} is slightly different in superficial
appearance, but it is not hard to verify that the dyadic version stated here is equivalent to
the condition in \cite{DS2}.  We note that we shall not need the full strength of this equivalence here,
but only the fact that our version of BAUP implies the version in \cite{DS2}, and hence implies UR.

We shall also require a new characterization of UR sets
of co-dimension 1, which is related to the BAUP and its variants.
For a sufficiently large constant $K_0$ to be chosen (see Lemma \ref{l4.1} below), we set
\begin{equation}\label{eq2.bstar}
B_Q^*:= B(x_Q,K_0^2\ell(Q))\,, \qquad \Delta^*_Q:= B_Q^*\cap E\,.
\end{equation}
Given a small positive number $\eps$, which we shall typically assume to be much smaller than $K_0^{-6}$,
we also set
\begin{equation}\label{eq2.bstarstar}
B_Q^{**}=B_Q^{**}(\eps) := B(x_Q,\eps^{-2}\ell(Q))\,,\quad
B_Q^{***}=B_Q^{***}(\eps) := B(x_Q,\eps^{-5}\ell(Q))\,. 
\end{equation}

\begin{definition} ({\bf ``$\eps$-local WHSA"})\label{def2.13} Given $\eps>0$,
we shall say that $Q\in\dd(E)$ satisfies the
$\eps$-{\it local WHSA} condition
(or more precisely, the ``$\eps$-local WHSA with parameter $K_0$") if there is a half-space
$H = H(Q)$, a hyperplane  $P=P(Q) =\partial H$, and a fixed positive number $K_0$
satisfying
\begin{enumerate}
\item $\dist(Z,E)\leq\eps\ell(Q),$ for every $Z\in P\cap B_Q^{**}(\eps)$.

\smallskip

\item $\dist(Q,P)\leq K_0^{3/2} \ell(Q).$

\smallskip

\item $H\cap B_Q^{**}(\eps)\cap E=\emptyset.$

\end{enumerate}
\end{definition}

Let us note that 
in particular, part (2) of the previous definition says that 
the hyperplane $P$ has an ``ample'' intersection with the ball  $B_Q^{**}(\eps)$. Indeed,
\begin{equation}\label{intersect-WHSA}
\dist(x_Q,P)\,\lesssim \,K_0^{\frac32}\,\ell(Q) 
\,\ll\eps^{-2}\ell(Q).
\end{equation}

\begin{definition} ({\bf ``WHSA"})\label{def2.14}
We shall say that an $n$-dimensional ADR set $E\subset \ree$
satisfies the {\it Weak  Half-Space Approximation} property
(``WHSA") if for some pair of positive constants $\eps_0$ and $K_0$, and for every positive
$\eps<\eps_0$, there is a constant $C_1=C_1(\eps)$ such that the set $\B$ of bad cubes in $\dd(E)$, for which the
$\eps$-local
WHSA condition with parameter $K_0$
fails, satisfies the packing condition
\begin{equation}\label{eq2.pack2}
\sum_{Q\subset Q_0,\, Q\in\B} \sigma(Q)\,\leq \,C_1 \,\sigma(Q_0)\,,\qquad \forall\, Q_0\in \dd(E)\,.
\end{equation}
\end{definition}

Next, we develop some further notation and terminology.
Let $\W$ be a fixed collection of closed Whitney cubes for an open set $\Omega$ with ADR boundary $E=\pom$,
and given $Q\in \dd(E)$, for the same constant $K_0$ as in \eqref{eq2.bstar}, we set

\begin{equation}\label{eq2.1}
\W_Q:= \left\{I\in \W:\,K_0^{-1} \ell(Q)\leq \ell(I)
\leq K_0\,\ell(Q),\, {\rm and}\, \dist(I,Q)\leq K_0\, \ell(Q)\right\}\,.
\end{equation}

We fix a small, positive parameter $\tau$, to be chosen momentarily, and given $I\in\W$,
we let
\begin{equation}\label{eq2.3*}I^* =I^*(\tau) := (1+\tau)I
\end{equation}
denote the corresponding ``fattened" Whitney cube.
We now choose $\tau$ sufficiently small that the cubes $I^*$ will retain the usual properties of Whitney cubes,
in particular that
$$\diam(I) \approx \diam(I^*) \approx \dist(I^*,E) \approx \dist(I,E)\,.$$

We then define Whitney regions 
with respect to $Q$ by setting
\begin{equation}\label{eq2.3}
U_Q:= \bigcup_{I\in \W_Q}I^*\,. 
\end{equation}
We observe that these Whitney regions may have more than one connected component,
but that the number of distinct components is uniformly bounded, depending only upon $K_0$ and dimension.
We enumerate the components of $U_Q$
as $\{U_Q^i\}_i$. 

Moreover, we enlarge the Whitney regions as follows.
\begin{definition}\label{def2.11a} For $\eps>0$,
and given $Q\in\dd(E)$,
we write $X\approx_{\eps,Q} Y$ if $X$ may be connected to $Y$ by a chain of
at most $\eps^{-1}$ balls of the form $B(Y_k,\delta(Y_k)/2)$, with
$\eps^3\ell(Q)\leq\delta(Y_k)\leq \eps^{-3}\ell(Q)$.
Given a sufficiently small parameter $\eps>0$, we then set
\begin{equation}\label{eq2.3a}
\tU^i_Q:= \left\{X \in\ree\setminus E:\, X\approx_{\eps,Q} Y\,,\, {\rm for\, some\,} Y\in U^{i}_Q\right\} \,.
\end{equation}
\end{definition}
\begin{remark}\label{r2.5}
Since $\tU^i_Q$ is
an enlarged version of $U_Q$, it may be that 
there are some $i\neq j$ for which
$\tU^i_Q$ meets $\tU^j_Q$.  This overlap will be harmless.
\end{remark}

\begin{lemma}[Bourgain \cite{B}]\label{Bourgainhm}  Suppose that
$\partial \Omega$ is $n$-dimensional ADR.  Then there are uniform constants $c\in(0,1)$
and $C\in (1,\infty)$, depending only on $n$ and ADR,
such that for every $x \in \partial\Omega$, and every $r\in (0,\diam(\partial\Omega))$,
if $Y \in \Omega \cap B(x,cr),$ then
\begin{equation}\label{eq2.Bourgain1}
\omega^{Y} (\Delta(x,r)) \geq 1/C>0 \;.
\end{equation}
\end{lemma}
We refer the reader to \cite[Lemma 1]{B} for the proof.  We note for future reference that
in particular,  if $\hat{x}\in \pom$
satisfies
$|X-\hat{x}|=\delta(X)$, and
$\Delta_X:= \pom\cap B\big(\hat{x}, 10\delta(X)\big)$,
then for a slightly different uniform constant $C>0$,
\begin{equation}\label{eq2.Bourgain2}
\omega^{X} (\Delta_X) \geq 1/C \;.
\end{equation}
Indeed, the latter bound follows immediately from \eqref{eq2.Bourgain1},
and the fact that we can form a Harnack Chain connecting
$X$ to a point $Y$ that lies on the line segment from $X$ to $\hat{x}$, and satisfies $|Y-\hat{x}|= c\delta(X)$.

As a consequence of Lemma \ref{Bourgainhm}, we have the following:

\begin{corollary}[\cite{HMT}]\label{cor2.12} Let
$\partial \Omega$ be $n$-dimensional ADR.   
Suppose that
$u\geq 0$ is harmonic in $\Omega\cap B(x,2r)$, and vanishes continuously on the surface ball
$\Delta(x,2r) = B(x,2r)\cap\pom$, with $x\in \pom$, and $0<r<\diam \pom$.  Then for some $\alpha >0$,
\begin{equation}\label{eq2.13}
u(Y) \leq C \left(\frac{\delta(Y)}{r}\right)^\alpha \frac1{|B(x,2r)|}\,\iint_{B(x,2r)\cap\Omega} u\,,\qquad \forall\, Y\in B(x,r)\cap\Omega\,,
\end{equation}
where the constants $C$ and $\alpha$ depend
only on dimension and the ADR constants for $\pom$.
\end{corollary}

\begin{lemma}[\cite{HMT}] \label{lemma2.green}
Let $\Omega$ be an open set with $n$-dimensional ADR boundary. There are positive, finite constants $C$, depending only on dimension, $\Lambda$
and $c_\theta$, depending on dimension, $\Lambda$, and $\theta \in (0,1),$
such that the Green function satisfies
\begin{eqnarray}\label{eq2.green}
&G(X,Y) \leq C\,|X-Y|^{1-n}\,;\\[4pt]\label{eq2.green2}
& c_\theta\,|X-Y|^{1-n}\leq G(X,Y)\,,\quad {\rm if } \,\,\,|X-Y|\leq \theta\, \delta(X)\,, \,\, \theta \in (0,1)\,;
\end{eqnarray}
\begin{equation}
\label{eq2.green-cont}
G(X,\cdot)\in C(\overline{\Omega}\setminus\{X\}) \qquad \mbox{and}\qquad G(X,\cdot)\big|_{\pom}\equiv 0\,,\qquad \forall X\in\Omega;
\end{equation}
\begin{equation}
\label{eq2.green3}
G(X,Y)\geq 0\,,\qquad \forall X,Y\in\Omega\,,\, X\neq Y;
\end{equation}
\begin{equation}\label{eq2.green4}
G(X,Y)=G(Y,X)\,,\qquad \forall X,Y\in\Omega\,,\, X\neq Y;
\end{equation}
and for every $\Phi \in C_0^\infty(\ree)$,
\begin{equation}\label{eq2.14}
\int_{\partial\Omega} \Phi\,d\omega^X -\Phi(X)
=
-\iint_\Omega
\nabla_Y G(Y,X) \cdot\nabla\Phi(Y)\, dY, \quad
\mbox{for a.e. }X\in\Omega.
\end{equation}
\end{lemma}

Next we present a version of one of the estimates obtained 
by Caffarelli-Fabes-Mortola-Salsa in \cite{CFMS}, which remains true even in the absence of connectivity:

\begin{lemma}[CFMS]\label{l2.10}
Suppose that $\partial \Omega$ is $n$-dimensional ADR.  
For every
$Y\in \Omega$ and $X\in\Omega$ such that $|X-Y|\ge \delta(Y)/2$ we have
\begin{equation}\label{eqn:right-CFMS}
\frac{G(Y,X)}{\delta(Y)}
\le
C\,\frac{\hm^X( \Delta_Y)}{\sigma( \,\Delta_Y)},
\end{equation}
where $\Delta_Y=B(\hat{y},10\delta(Y))$ with $\hat{y}\in\pom$ such that $|Y-\hat{y}|=\delta(Y)$.
\end{lemma}

For future use, we note that as a consequence of \eqref{eqn:right-CFMS}, it follows directly
that for 
every $Q\in\dd(\pom)$, 
if $Y\in C_0 B_Q$ with $\delta(Y)\gtrsim C_0'\ell(Q)$, 
then there exists $C=C(n,ADR,C_0,C_0')$ such that
\begin{equation}\label{eqn:right-CFMS:cubes}
\frac{G(Y,X)}{\ell(Q)}
\lesssim \,
\,\frac{\hm^X(C Q)}{\sigma(C Q)} \,
\lesssim \, \fint_Q \M(k^X1_{CQ}) \, d\sigma,
\qquad \forall\,X\notin C\,B_Q\,,
\end{equation}
where $\M$ is the usual Hardy-Littlewood maximal operator on 
$\pom$, and $k^X$ is the Poisson kernel for $\om$ with pole at $X$.


\begin{proof}[Proof of Lemma \ref{l2.10}]
We follow the 
well-known argument of  \cite{CFMS} (see also \cite[Lemma 1.3.3]{Ke}). Fix $Y\in \Omega$ and 
write $B^Y=\overline{B(Y,\delta(Y)/2)}$. Consider the open set $\widehat{\Omega}=\Omega\setminus B^Y$ 
for which clearly $\partial\widehat{\Omega}=\pom\cup \partial B^Y$.   Set 
$$u(X):=G(Y,X)/\delta(Y)\,,\qquad v(X):=\hm^X(\Delta_Y)/\sigma( \Delta_Y)\,,$$
for every $X\in \widehat{\Omega}$. Note 
that both $u$ and $v$ are non-negative harmonic functions in 
$\widehat{\Omega}$. If $X\in \pom$ then $u(X)=0\le v(X)$. Take now $X\in\partial B^Y$ 
so that $u(X)\lesssim \delta(Y)^{-n}$ by \eqref{eq2.green}. On the other 
hand, if we fix $X_0\in\partial B^Y$ with $X_0$ on the line
segment that joints $Y$ and $\hat{y}$, then $2\Delta_{X_0} = \Delta_Y$, so that
$v(X_0) \gtrsim \delta(Y)^{-n}$, by  \eqref{eq2.Bourgain2}.  By Harnack's inequality,
we then obtain $v(X) \gtrsim \delta(Y)^{-n}$, for all $X\in \partial B^Y$.
Thus, $u\lesssim v$ in $\partial\widehat{\Omega}$ and by the
maximum principle this immediately extends to $\widehat{\Omega}$ as desired.
\end{proof}

\begin{lemma}\label{lemma:G-aver}
Suppose that $\pom$ is $n$-dimensional ADR. Let $B=B(x,r)$ with $x\in\pom$ and $0<r<\diam(\pom)$, and set $\Delta=B\cap \pom$. There exists $\kappa_0>2$ depending only in $n$ and the ADR constant of $\pom$ such that
for $X\in\Omega\setminus \kappa_0\,B$  we have
\begin{equation}\label{eqn:aver-B}
\sup_{\frac12B} G(\cdot,X)\,\lesssim\,\frac1{|B|}\iint_B G(Y,X)\,dY
\,\le \,
C\,r\,\frac{\hm^X(C\Delta)}{\sigma(C\,\Delta)}.
\end{equation}
where $C$ depends in $n$ and the ADR constant of $\pom$.
\end{lemma}

\begin{proof} Extending $G(\cdot,X)$ to be 0 outside of $\om$, we obtain a sub-harmonic function in
$B$.  The first inequality in \eqref{eqn:aver-B} follows immediately.  To prove the second inequality,
we set $\Sigma_B=\{I\in \W: I\cap B\neq\emptyset\}$ and note that if $I\in\Sigma_B$ then 
$$\ell(I)\approx \dist(I,\pom)\le \dist(I,x)\le r\,.$$ 
In particular we can find $\kappa_0$ depending only in the implicit constants in the 
previous estimate so that $d(X,4\,I)\ge 4\,r$ for every $I\in \Sigma_B$. Let $Q_I\in\dd$ be so 
that $\ell(Q_I)=\ell(I)$ and $\dist(I,\pom)=\dist(I,Q_I)$. Note that then $\ell(Q_I)\lesssim r$ and 
$Y(I)$, the center of $I$, satisfies 
$Y(I)\in CB_{Q_I}$ and $\delta(Y(I))\approx \ell(I)\approx \ell(Q_I)$. Hence we 
can invoke \eqref{eqn:right-CFMS:cubes} (taking $\kappa_0$ larger if needed) and obtain that
for every $Y\in I$,
$$
G(Y,X)\approx G(Y(I),X)
\lesssim \ell(I)\,\frac{\hm^X(C Q_I)}{\sigma(C Q_I)},
$$
where the first estimate uses Harnack's inequality in $2I\subset \Omega$.
Hence,
\begin{multline*}
\iint_B G(Y,X)\,dY
\le
\sum_{I\in\Sigma_B} \iint_I G(Y,X)\,dY
\lesssim
\sum_{I\in\Sigma_B} \ell(I)^2\,\hm^X(C Q_I)
\\
\le
\sum_{k:2^{-k}\lesssim r}2^{-2\,k}\sum_{I\in\Sigma_B: \ell(I)=2^{-k}} \hm^X(C Q_I)
\lesssim
r^2\,\hm^X(C' \Delta).
\end{multline*}
where in the last inequality we have used that the cubes $Q_I$ have uniformly bounded overlap whenever $\ell(I)=2^{-k}$ and they are all contained in $C'\Delta$. This and the ADR property readily yields the desired estimate.

\end{proof}

\section{Proof of Theorem \ref{t1}, step 1: preliminary arguments}\label{s4}


Let us introduce some notation.  Fix $Q_0\in\dd(\pom)$. Recall \eqref{cube-ball2} and take $B_{Q_0}=B(x_{Q_0}, r_{Q_0})$ with $r_{Q_0}\approx\ell(Q_0)$ so that  $\Delta_{Q_0}=B_{Q_0} \cap \pom\subset Q_0$. Let $X_0$ be the associated ``corkscrew'' point given in the statement of Theorem \ref{t1}. In particular, for $C_1$ large enough (so that $2\,Q_0\subset C_1\Delta_{Q_0}$) we have that $X_0\in C_0 B_{Q_0}$ with $\delta(X_0)\ge C_0^{-1}\,r_{Q_0}\approx \ell(Q_0)$ for which $\hm^{X_0}\ll \sigma$ in $2Q_{0}$ and moreover
\begin{equation}\label{eqn:Bour+Ainfty}
\omega^{X_0}(Q_0)\ge C_0^{-1},
\qquad\qquad
\int_{2Q_0} k^{X_0}(y)^p\,d\sigma
\le
C_0\,\sigma(2\,Q_0)^{1-p}
\end{equation}
Set $\hm := C_0\, \sigma(Q_0)\, \hm^{X_{0}}$, and let $k:=  d\hm/d\sigma$ be the corresponding
normalized Poisson kernel.    We then have
\begin{equation}\label{eq4.6}
1\, \leq\,\frac{\hm(Q_0)}{\sigma(Q_0)}\,\leq \,\frac{\hm(\pom)}{\sigma(Q_0)}\le C_0
\end{equation}
and
\begin{equation}\label{eq4.6-bis}
\left(\fint_{2\,Q_0} k(y)^p\,d\sigma(y)\right)^{\frac1p}
\le
C_0^{1+\frac1p}.
\end{equation}
Given a family $\F=\{Q_j\}$ of disjoint sub-cubes of $Q_0$, we set 
$$\dd_{\F,Q_0}:=\{ Q\subset Q_0: \, Q {\rm \,\, is\, not\, contained\, in\, any}\,\, Q_j\in\F\}\,,$$
and for any $Q\in \dd(\pom)$, we set
$$\dd_Q:= \{Q'\subset Q\}\,.$$

As above, let $\M$ denote the usual Hardy-Littlewood maximal operator on $\pom$.  

\begin{lemma}\label{l4.4}
Under the assumptions of Theorem \ref{t1} and following the previous notation, there is a pairwise disjoint family $\F=\{Q_j\}_j\setminus\{Q_0\}$, such that
\begin{equation}\label{eq4.8}
\sigma\left(Q_0\setminus \left(\cup_j Q_j\right)\right) \geq c_0\,\sigma(Q_0)\,
\end{equation}
where  $0<c_0\le 1$ depends only on the implicit constants in the hypotheses of Theorem \ref{t1}, and
\begin{equation}\label{eq4.9}
\frac12\, \leq \,\frac{\hm(Q)}{\sigma(Q)}\,\leq \fint_Q \M\big(k1_{2_{Q_0}}\big)\,d\sigma\,
\le\, C\,,\qquad \forall\, Q\in\dd_{\F,Q_0}\,,
\end{equation}
where $C>1$ depends only on the implicit constants in the hypotheses.
\end{lemma}

\begin{proof}
The proof is based on a stopping time argument similar to those used in the proof of the
Kato square root conjecture \cite{HMc},\cite{HLMc}, \cite{AHLMcT}, and in local $Tb$ theorems.  We begin by noting that
\begin{multline}\label{eq4.11}
\fint_{Q_0} \M\big(k1_{2\,Q_0}\big)\,d\sigma\,
\leq
\left(\fint_{Q_0} \left(\M\big(k 1_{2\,Q_0}\big)\right)^p\,d\sigma\right)^{1/p}
\\
\le C_1\left(\fint_{2\,Q_0}k^p\, d\sigma\right)^{1/p}
\,\le C_1\, C_0^{1+\frac1p}
=:C_2,
\end{multline}
where we have used the $L^p(\sigma)$ boundedness of the Hardy-Littlewood maximal function and \eqref{eq4.6-bis}.
We now let $\F=\{Q_j\}\subset \dd_{Q_0}$ be the collection of sub-cubes that are maximal
with respect to the property that either
\begin{equation}\label{eq4.12}
\frac{\hm(Q_j)}{\sigma(Q_j)}\,< \,\frac12\,,
\end{equation}
and/or
\begin{equation}\label{eq4.13}
\fint_{Q_j} \M\big(k 1_{2\,Q_0}\big)\,d\sigma\,>\, C_2\,K\,,
\end{equation}
where $K\ge 1$ is a sufficiently large number to be chosen momentarily.  Note that
$\F=\{Q_j\}\subset \dd_{Q_0}\setminus\{Q_0\}$ by \eqref{eq4.6} and \eqref{eq4.11}.
We shall say that
$Q_j$ is of ``type I" if \eqref{eq4.12} holds, and $Q_j$ is of ``type II"
if \eqref{eq4.13} holds but \eqref{eq4.12} does not.  Set $A:= Q_0\setminus (\cup_j Q_j)$,
and $F:= \cup_{Q_j  {\rm \,type\, II}}\, Q_j$.
Then by \eqref{eq4.6},
\begin{equation}\label{eq4.13a}
\sigma(Q_0)\leq\hm(Q_0) = \, \sum_{Q_j  {\rm \,type\, I}} \hm(Q_j)
\,+\,\hm(F)\,+\,\hm(A)\,.
\end{equation}
By definition of the type I cubes,
\begin{equation}\label{eq4.14}
\sum_{Q_j  {\rm \,type\, I}} \hm(Q_j)\leq \frac12\sum_j\sigma(Q_j) \leq \frac12 \sigma(Q_0)\,.
\end{equation}
To handle the remaining terms, observe that
\begin{multline}\label{eq4.16}
\sigma(F) = \sum_{Q_j  {\rm \,type\, II}}\, \sigma(Q_j)\,
\le
\,
\frac1{C_2\,K}\, \sum_j \int_{Q_j} \M\big(k1_{2\, Q_0}\big)\,d\sigma
\\[4pt]
\leq \, \frac1{C_2\,K}\,  \int_{Q_0} \M\big(k1_{2\,Q_0 }\big)\,d\sigma\,
\leq\,\,\frac1K\,\sigma(Q_0)\,,
\end{multline}
by the definition of the type II cubes and \eqref{eq4.11}.
Combining \eqref{eq4.6-bis} and \eqref{eq4.16},
we find that
\begin{equation}\label{eq4.17}
\hm(F) =\int_F k\,d\sigma
\le
\left(\int_{2\,Q_0} k^p\,d\sigma\right)^{\frac1p}
\,\sigma(F)^{\frac1{p'}}
\, \leq\,
C_0^{1+\frac1p}
K^{-1/p'}\sigma(2\,Q_0) \,
\le
\,\frac14\,\sigma(Q_0)\,,
\end{equation}
by choice of $K$ large enough.
By \eqref{eq4.14} and \eqref{eq4.17}, we may hide the two small terms on the left hand side of
\eqref{eq4.13a}, and then use \eqref{eq4.6-bis}, to obtain
$$
\sigma(Q_0) \leq \,4\,\hm(A)\, =\, 4\int_A k \, d\sigma\,
\lesssim
\left(\int_{2\,Q_0} k^p\,d\sigma\right)^{\frac1p}
\,\sigma(A)^{\frac1{p'}}
\lesssim
\,\sigma(A)^{1/p'}\sigma(Q_0)^{\frac1p}\,.
$$
Estimate \eqref{eq4.8} now follows readily.  Moreover, \eqref{eq4.9} holds, by the maximality of
the cubes $Q_j$, and our choice of $K$.
\end{proof}

We recall that the ball $B_Q^*$ and surface ball $\Delta^*_Q$ are defined in \eqref{eq2.bstar}.

\begin{lemma}\label{l4.1}
Under the notation of Lemma \ref{l4.4}, if the constant $K_0$ in \eqref{eq2.1} is chosen sufficiently large,
for each $Q\in\dd_{\F, Q_0}$ with $\ell(Q)\le K_0^{-1}\,\ell(Q_0) $ there exists $Y_{Q}\in U_Q$ with $\delta(Y_Q)\le |Y_Q-x_Q|\lesssim \ell(Q)$ (where the implicit constant is independent of $K_0$) such that
\begin{equation}\label{eq4.2}
\frac{\hm^{X_{0}}(Q)}{\sigma(Q)}\,\leq C\,|\nabla G(X_{0},Y_{Q})|,
\end{equation}
where $C$ depends on $K_0$ and the implicit constants in the hypotheses of Theorem \ref{t1}.
\end{lemma}

\begin{remark}\label{remark:X0-YQ}
Our choice of $K_0$ in the previous result will guarantee that $\delta(X_0)\approx \ell(Q_0)\ge K_0^{-1/2}\,\ell(Q_0)$. Note also that the point $Y_Q$ is an effective corkscrew relative to $Q$ since $\delta(Y_Q)\gtrsim K_0^{-1}\ell(Q)$ (as $Y\in U_Q$) and also $|Y_Q-x_Q|\lesssim \ell(Q)$ (with constant independent of $K_0$). Abusing the notation we will say that $Y_Q$ is a corkscrew point relative to $Q$ and we observe that the corresponding constant depends on $K_0$.
\end{remark}

\begin{proof}
Fix $Q\in\dd_{\F, Q_0}$ with $\ell(Q)\le K_0^{-1}\,\ell(Q_0) $ and $K_0$ large enough to be chosen. Recall that $\delta(X_0)\approx \ell(Q_0)$ and therefore if $K_0$ is large enough we may assume that $\delta(X_0)\approx \ell(Q_0)\ge K_0^{-1/2}\,\ell(Q_0)$. Recall \eqref{cube-ball} and write $\hat{B}_Q=B(x_Q, \hat{r}_Q)$, $\hat{\Delta}_Q=\hat{B}_Q\cap E$ so that $\hat{r}_{Q}\approx \ell(Q)$ and $Q\subset\frac12\,\hat{\Delta}_Q$. Let $0\le \phi_Q\in C_0^\infty(\hat{B}_Q)$ so that $\phi_Q\equiv 1$ in $\frac12\,\hat{B}_Q$ and $\|\nabla \phi_Q\|\lesssim \ell(Q)^{-1}$. Note that
$$
K_0^{1/2}\,\ell(Q)
\le
K_0^{-1/2}\,\ell(Q_0)
\lesssim\delta(X_0)\le |X_0-x_Q|
$$
which implies that $X_0\notin 4\,\hat{B}_Q$ provided $K_0$ is large enough. We can next use \eqref{eq2.14} (if needed we can slightly move $X_0$ so that the equality holds at $X_0$ and then use Harnack's inequality when needed to move back to $X_0$, details are left to the interested reader). Then
\begin{multline}\label{eqn:estw}
\ell(Q)\,\omega^{X_0}(Q)
\le
\ell(Q)\,\int_{\pom} \phi_Q\,d\omega^{X_Q} \lesssim \iint_{\hat{B}_Q\cap\Omega} |\nabla G(X_0,Y)|\,dY
\\
\le
\iint_{\hat{B}_Q\cap U_Q} |\nabla G(X_0,Y)|\,dY + \iint_{(\hat{B}_Q\cap \Omega )\setminus U_Q} |\nabla G(X_0,Y)|\,dY
=
:\mathcal{I}+\mathcal{II}.
\end{multline}

Notice that by construction $(\hat{B}_Q\cap \Omega )\setminus U_Q\subset \{Y\in \hat{B}_Q: \delta(Y)\lesssim K_0^{-1}\,\ell(Q) \}$. Writing  $\Sigma_Q:=\{I\in \W: I\cap \hat{B}_Q\neq\emptyset, \ell(I)\lesssim K_0^{-1}\,\ell(Q)\}$ and using interior estimates (note that $4\,I\subset\Omega$) we have
$$
\mathcal{II}
\le
\sum_{I\in \Sigma_Q} \iint_I |\nabla G(X_0,Y)|\,dY
\lesssim
\sum_{I\in \Sigma_Q} \ell(I)^n\, G(X_0,Y_I)
$$
where $Y_I$ is the center of $I$. We next use Corollary \ref{cor2.12} and Lemma \ref{lemma:G-aver} (we may need to take $K_0$ larger) observe that
\begin{multline}\label{eqn:G-CKS-I}
G(X_0,Y_I)\lesssim
\left(\frac{\ell(I)}{\ell(Q)}\right)^\alpha\,\frac1{|2\,\hat{B}_Q|}\iint_{2\,\hat{B}_Q\cap \Omega} G(X_0,Y)\,dY
\lesssim
\left(\frac{\ell(I)}{\ell(Q)}\right)^\alpha\,
\ell(Q)\,\frac{\hm^{X_0}(C\,Q)}{\sigma(C\,Q)}
\\
\le
\sigma(Q_0)^{-1}\left(\frac{\ell(I)}{\ell(Q)}\right)^\alpha\,\ell(Q)
\fint_Q \M\big(k1_{2_{Q_0}}\big)\,d\sigma
\lesssim
\sigma(Q_0)^{-1}
\left(\frac{\ell(I)}{\ell(Q)}\right)^\alpha\,\ell(Q),
\end{multline}
where we recall that $\hm=C_0\,\sigma(Q_0)\,\hm^{X_0}$, we have used that if $K_0$ is large enough and $Q\in\dd_{Q_0}$ with $\ell(Q)\le K_0^{-1}\ell(Q_0)$ then $C\,Q\subset 2\,Q_0$,  and we have also invoked \eqref{eq4.9} (recall that $Q\in \dd_{\F,Q_{0}}$). As before let $Q_I\in\dd$ be so that $\ell(Q_I)=\ell(I)$ and $\dist(I,E)=\dist(I,Q_I)$ and
observe that since $I$ meets $\hat{B}_Q$ then $Q_I\subset C\,Q$ for some uniform constant. Let us observe that from the Properties of the Whitney cubes, for every $k$, the family of cubes $\{Q_I\}_{\ell(I)=2^{-k}}$ has bounded overlap uniformly in $k$. Hence we can use \eqref{eqn:G-CKS-I}, \eqref{eq4.9}, and the definition of $\omega$ to obtain
\begin{align*}
\mathcal{II}
&\lesssim
\sigma(Q_0)^{-1}
\sum_{I\in \Sigma_Q} \ell(I)^n\, \left(\frac{\ell(I)}{\ell(Q)}\right)^\alpha\,\ell(Q)
\\
&\lesssim
\sigma(Q_0)^{-1}
\ell(Q)^{1-\alpha}
\sum_{k:2^{-k}\lesssim K_0^{-1}\,\ell(Q)}2^{-k\,\alpha}\sum_{I\in \Sigma_Q:\ell(I)=2^{-k}}\sigma(Q_I)
\\
&\lesssim
\sigma(Q_0)^{-1}
K_0^{-\alpha}\,\ell(Q)\,\sigma(C\,Q)
\\
&
\le
\frac12\,\ell(Q)\,\hm^{X_0}(Q),
\end{align*}
provided $K_0$ is large enough. We can then plug this estimate in \eqref{eqn:estw} and hide it to obtain that
\begin{multline*}
\ell(Q)\,\omega^{X_0}(Q)
\le \mathcal{I}
\le
\iint_{\hat{B}_Q\cap U_Q} |\nabla G(X_0,Y)|\,dY
=
\sum_i \iint_{\hat{B}_Q\cap U_Q^i} |\nabla G(X_0,Y)|\,dY
\\
\lesssim
\ell(Q)^{n+1}\max_i \sup_{Y\in \hat{B}_Q\cap U_Q^{i_0}}|\nabla G(X_0,Y)|
\lesssim
\ell(Q)\,\sigma(Q)\max_i \sup_{Y\in \hat{B}_Q\cap U_Q^{i_0}}|\nabla G(X_0,Y)|,
\end{multline*}
where we recall that number of components of $U_Q$ is uniformly bounded. This clearly implies that we can find $Y_Q\in \hat{B}_Q\cap U_Q^{i}$ for some $i$
such that  $\omega^{X_0}(Q)/\sigma(Q)\lesssim |\nabla G(X_0,Y_Q)|$. To complete the proof we simply observe that $\delta(Y_Q)\le |Y_Q-x_Q|\le \hat{r}_Q\lesssim\ell(Q)$.
\end{proof}

\section{Proof of Theorem \ref{t1}, step 2:   the Lewis-Vogel argument}\label{s6}

We recall that $B_Q^{***}(\eps)= B(x_Q,\eps^{-5}\ell(Q))$, as in \eqref{eq2.bstarstar}.
Set $\Delta_{Q}^{***}(\eps):= E\cap B_{Q}^{***}(\eps)$.

%

Our proof here is a refinement/extension of the arguments in \cite{LV}, who, as mentioned in the introduction, treated
the special case that the Poisson kernel $k\approx 1$.
Our goal in this section is to show that $E=\pom$ satisfies WHSA, and hence is UR, by Proposition \ref{prop2.20}.
Turning to the details, we fix $Q_0\in\dd(\pom)$, and we set
$$
u(Y):=
C_0 \,\sigma(Q_0)\,G(X_0,Y)\,,
$$
where $X_0$ is the ``corkscrew'' associated with $\Delta_{Q_0}$ (see the beginning of Section \ref{s4}) and $C_0$ is the constant in \eqref{eqn:Bour+Ainfty}.
As above, for the same constant $C_0$, we set
$$\hm := \,C\,\sigma(Q_0) \,\hm^{X_0}\,,$$
and we recall that by \eqref{eq4.6},
\begin{equation}\label{eq6.1}
\frac{\hm(Q_0)}{\sigma(Q_0)} \approx 1\,.
\end{equation}

Let $\F=\{Q_j\}_j$ be the family of
maximal stopping time cubes constructed in Lemma \ref{l4.4}.
Combining \eqref{eq4.2} and \eqref{eq4.9}, 
we see that
\begin{equation}\label{eq6.5}
 |\nabla u(Y_Q)| \gtrsim \, 1\,,\qquad \forall \, Q\in \dd_{\F,Q_0}^*:=\{Q\in \dd_{\F,Q_0}: \ \ell(Q)\le K_0^{-1}\,\ell(Q_0)\}\,,
\end{equation}
where $Y_Q\in U_Q$ is the point constructed in Lemma \ref{l4.1}.  We recall
that the  Whitney region $U_{Q}$ has a uniformly bounded number
of connected components, which we have enumerated as $\{U^i_{Q}\}_i$.
We now fix the particular $i$
such that  $Y_Q\in U^i_Q\subset \tU^i_Q$, where the latter is the enlarged Whitney region
constructed in Definition \ref{def2.11a}.

For a suitably small $\eps_0$, say $\eps_0\ll K_0^{-6}$,
we fix an arbitrary positive $\eps<\eps_0$, and we fix also a large positive number $M$
to be chosen.
For each point $Y\in\Omega$, we set
\begin{equation}\label{eq5.1}
B_Y:= \overline{B\big(Y,(1-\eps^{2M/\alpha})\delta(Y)\big)}\,,\qquad \widetilde{B}_Y:= \overline{B\big(Y,\delta(Y)\big)}\,,
\end{equation}
where $\alpha>0$ is the DG/N exponent at the boundary (see Corollary \ref{cor2.12}).

For $Q\in\dd_{\F,Q_0}$,  we consider three cases.

\noindent{\bf Case 0}: $Q\in\dd_{\F,Q_0}$, with $\ell(Q)> \eps^{10}\,\ell(Q_0)$.

\smallskip
\noindent{\bf Case 1}: $Q\in \dd_{\F,Q_0}$, with $\ell(Q)\le \eps^{10}\,\ell(Q_0)$
(in particular $Q\in \dd_{\F,Q_0}^*$;  see \eqref{eq6.5}), and
\begin{equation}\label{eq5.3}
\sup_{X,Y\in \tU^i_{Q}}\,\sup_{Z_1\in B_Y,\,Z_2\in B_X} |\nabla u(Z_1)
 -\nabla u(Z_2)|\,>\, \eps^{2M}\,.
\end{equation}

\smallskip
\noindent{\bf Case 2}: $Q\in \dd_{\F,Q_0}$, with $\ell(Q)\le \eps^{10}\,\ell(Q_0)$ 
(in particular $Q\in \dd_{\F,Q_0}^*$), and
\begin{equation}\label{eq5.2}
\sup_{X,Y\in \tU^i_{Q}}\,\sup_{Z_1\in B_Y,\,Z_2\in B_X} |\nabla u(Z_1)
 -\nabla u(Z_2)|\,\leq\, \eps^{2M}\,.
\end{equation}

We trivially see that the cubes in Case 0 satisfy a packing condition:
\begin{equation}\label{eqn-case0-pack}
\sum_{\substack{Q\in \dd_{\F,Q_0} \\ {\rm Case\, 0 \, holds}}}\sigma(Q)
\le
\sum_{Q\in\dd_{Q_0},\, \ell(Q)>\eps^{10}\,\ell(Q_0)} \sigma(Q)
\lesssim
(\log \eps^{-1})\,\sigma(Q).
\end{equation}

Before proceeding further, 
let us first note that if $\ell(Q)\leq \eps^{10}\ell(Q_0)$, then by \eqref{eq6.5},  \eqref{eqn:right-CFMS:cubes}
(which we may apply with $X=X_0$, since $\ell(Q)\ll\ell(Q_0))$, and \eqref{eq4.9},
\begin{equation}\label{eq6.9*}
1\lesssim  |\nabla u(Y_Q) |\lesssim \frac{u(Y_Q)}{\delta(Y_Q)}\lesssim 1\,.
\end{equation}

Next, we treat Case 1, and for these cubes, we shall also obtain a packing condition.
We now augment $\tU^i_{Q}$ as follows.
Set
\begin{equation*}
\W^{i,*}_{Q}:= \left\{I\in \W: I^*\,\, {\rm meets}\, B_Y\, {\rm for\, some\,} Y \in \left(\cup_{X\in\tU^i_{Q}}
B_X\right)\right\}
\end{equation*}
(and define $\W^{j,*}_{Q}$ analogously for all other $\tU^j_Q$), and set
$$
U^{i,*}_{Q} := \bigcup_{I\in\W^{i,*}_{Q}} I^{**}\,,\qquad U^{*}_{Q} := \bigcup_j U^{j,*}_{Q}
$$
where $I^{**}=(1+2\tau)I$ is a suitably fattened Whitney cube, with $\tau$ fixed as above.
By construction,
$$
\tU^i_{Q}\,
\subset\,
\bigcup_{X\in\tU^i_{Q}} B_X\,
\subset
\bigcup_{Y\in\cup_{X\in\tU^i_{Q}} B_X} B_Y
\subset\, U^{i,*}_{Q}\,,
$$
and for all $Y\in U_{Q}^{i,*}$, we have that $\delta(Y)\approx \ell(Q)$ (depending of course on $\eps$).
Moreover, also by construction, 
there is a Harnack path connecting any pair of points in 
$U^{i,*}_{Q}$ (depending again on $\eps$),
and furthermore, for every
$I\in\W^{i,*}_{Q}$ (or for that matter for every $I\in \W^{j,*}_{Q},\, j\neq i$),
$$
\eps^{s}\,\ell(Q)\lesssim \ell(I) \lesssim \eps^{-3}\,\ell(Q),
\qquad
\dist(I,Q)\lesssim \eps^{-4} \,\ell(Q)\,,
$$
where $0<s =s(M,\alpha)$.
Thus, by Harnack's
inequality and \eqref{eq6.9*}, 
\begin{equation}\label{eq6.9}
C^{-1} \delta(Y)\,\leq\,u(Y) \,\leq \,C \delta(Y)\,, \qquad \forall\, Y\in U^{i,*}_{Q}\,,
\end{equation}
with $C=C(K_0,\eps,M)$, where we have used that $u$ is a solution in $\om$ away from the pole at $X_0$,
and that $X_0$ is far from $U_Q^*$, since $\ell(Q)\ll\ell(Q_0)$.   Moreover, for future reference, we note that 
the upper bound for $u$ holds in all of $U_Q^*$, i.e.,
\begin{equation}\label{eq4.10a}
u(Y) \,\leq \,C \delta(Y)\,, \qquad \forall\, Y\in U^{*}_{Q}\,,
\end{equation}
by  \eqref{eqn:right-CFMS:cubes} and \eqref{eq4.9}, where again $C=C(K_0,\eps,M)$.
Choosing $Z_1,Z_2$ as in \eqref{eq5.3}, and then using the mean value property of harmonic functions,
we find that
$$\eps^{2M} \leq \,C_\eps\left(\ell(Q)\right)^{-(n+1)}\iint_{B_{Z_1}\cup\, B_{Z_2}}|\nabla u(Y) -\vec{\beta}| dY\,,$$
where $\vec{\beta}$ is a constant vector at our disposal.
In turn, by Poincar\'e's inequality (see, e.g., \cite[Section 4]{HM-I} in this context),
we obtain that
$$
\sigma(Q)\lesssim\,  \iint_{U^{i,*}_{Q}} |\nabla^2 u(Y)|^2 \delta(Y)\,dY
\lesssim
 \iint_{U^{i,*}_{Q}} |\nabla^2 u(Y)|^2 u(Y)\,dY\,,
 $$
where the implicit constants depend on $\eps$, and
in the last step we have used \eqref{eq6.9}.
Consequently,
\begin{multline}\label{eq6.11-ant}
\sum_{\substack{Q\in \dd_{\F,Q_0}\\ {\rm Case\, 1 \, holds} }}  \sigma(Q)
\lesssim
\sum_{\substack{Q\in \dd_{\F,Q_0}\\ \ell(Q)\leq\eps^{10}\ell(Q_0) }}
\iint_{U^{*}_Q} |\nabla^2 u(Y)|^2 u(Y)\,dY
\\[4pt]
\lesssim  \, 
\iint_{\Omega^{*}_{\F,Q_0} } |\nabla^2 u(Y)|^2 u(Y)\,dY,
 \end{multline}
where
\begin{equation}\label{sawtooth-appen-HMM}
\Omega^{*}_{\F,Q_0} := \interior\bigg(\bigcup_{\substack{Q\in \dd_{\F,Q_0}\\ \ell(Q)\leq\eps^{10}\ell(Q_0) }} 
U^{*}_Q\bigg),
\end{equation}
and where we have used that the enlarged Whitney regions $U^{*}_Q$ have bounded overlaps.

Take an arbitrary $N>1/\eps$ (eventually $N\to\infty$), and augment 
$\F$ by adding to it all subcubes $Q\subset Q_0$ with $\ell(Q)\leq 2^{-N}\,\ell(Q_0)$.  Let
$\F_N\subset\dd_{Q_0}$ denote the collection of maximal cubes of this augmented family. 
Thus, $Q\in\dd_{\F_N, Q_0}$ iff $Q\in\dd_{\F,Q_0}$ and $\ell(Q)>2^{-N}\,\ell(Q_0)$. Clearly, $\dd_{\F_N, Q_0}\subset \dd_{\F_{N'}, Q_0}$ if $N\le N'$ and therefore $\Omega^{*}_{\F_N,Q_0}\subset \Omega^{*}_{\F_{N'},Q_0}$ 
(where $\Omega^{*}_{\F_N,Q_0}$ is defined as in \eqref{sawtooth-appen-HMM} with 
$\F_N$ replacing $\F$). By monotone convergence and \eqref{eq6.11-ant}, we have that
\begin{equation}\label{eq6.11}
\sum_{\substack{Q\in \dd_{\F,Q_0}\\ {\rm Case\, 1 \, holds} }}  \sigma(Q)
\lesssim
\limsup_{N\to \infty}
\iint_{\Omega^{*}_{\F_N,Q_0} } |\nabla^2 u(Y)|^2 u(Y)\,dY.
 \end{equation}
It therefore suffices to establish bounds for the latter integral that are uniform in $N$, with $N$ large.

Let us then fix $N>1/\eps$.   Since $\Omega^{*}_{\F_N,Q_0}$ is a finite union of fattened
Whitney boxes,
 we may now integrate by parts, using the identity
 $2|\nabla \partial_k u|^2 = \dv \nabla (\partial_k u)^2$ for harmonic functions, to obtain that
 \begin{multline}\label{eq6.12}
  \int\!\!\int_{\Omega^{*}_{\F_N,Q_0}} 
 |\nabla^2 u(Y)|^2 u(Y)\,dY \lesssim \int_{\partial \Omega^{*}_{\F_N,Q_0}} 
 \left(|\nabla ^2 u|\, |\nabla u|\, u + |\nabla u|^3\right) dH^n
 \\[4pt]
 \leq \,C_\eps \,H^n(\partial \Omega^{*}_{\F_N,Q_0}) 
 \end{multline}
 where in the second inequality we have used 
the standard estimate
  $$\delta(Y)  |\nabla^2 u(Y)| \,\lesssim  \,|\nabla u(Y)|
\, \lesssim\, \frac{u(Y)}{\delta(Y) }\,,$$ 
 along with \eqref{eq4.10a}.  We observe that $\Omega^{*}_{\F_N,Q_0}$ is a sawtooth domain in the sense of
 \cite{HMM}, or to be more precise, it is a union of a bounded number, depending on $\eps$, of such sawtooths,
the maximal cube for each of which is a sub-cube of $Q_0$ with length on the order of $\eps^{10}\ell(Q_0)$.
Thus, by \cite[Appendix A]{HMM}, $\partial \Omega^{*}_{\F_N,Q_0}$ is ADR, uniformly in $N$, and therefore
$$H^n(\partial \Omega^{*}_{\F_N,Q_0}) \,\leq \,C_\eps \left(\diam(\partial \Omega^{*}_{\F_N,Q_0})\right)^n 
\,\leq\, C_\eps \,\sigma(Q_0)\,.$$
Combining the latter estimate with \eqref{eq6.11} and \eqref{eq6.12}, we obtain
\begin{equation}\label{eq6.13}
\frac{1}{ \sigma (Q_0)}
\sum_{Q\in \dd_{\F,Q_0}:\, {\rm Case\, 1 \, holds}}\sigma(Q) \leq C(\eps,K_0,M,\eta)\,.
\end{equation}

Now we turn to Case 2.  
We claim that for every $Q$ as in Case 2, the $\eps$-local WHSA property
(see Definition \ref{def2.13}) holds, provided that $M$ is taken large enough.
Momentarily taking this claim for granted,
we may complete
the proof of Theorem \ref{t1} as follows.  Given the claim, it follows that the cubes $Q\in \dd_{\F,Q_0}$,
which belong to the bad collection $\B$ of cubes in $\dd(\pom)$
for which the $\eps$-local WHSA condition fails, must be as in Case 0 or Case 1, and therefore, by \eqref{eqn-case0-pack} and \eqref{eq6.13}, satisfy the packing estimate
\begin{equation}\label{pack-saw}
\sum_{Q\in\B\cap \dd_{\F,Q_0}}\! \sigma(Q)\, \leq\, C_\eps \,\sigma(Q_0)\,.
\end{equation}
For each $Q_0\in\dd(\pom)$, there is a family $\F\subset\dd_{Q_0}$
for which \eqref{pack-saw}, and also the ``ampleness" condition \eqref{eq4.8}, hold uniformly.
We may therefore invoke a well known lemma of John-Nirenberg type to deduce that
\eqref{eq2.pack2} holds for all $\eps\in (0,\eps_0)$, 
and therefore to conclude that $\pom$ satisfies the WHSA condition 
(Definition \ref{def2.14}), 
and hence is UR,
by Proposition \ref{prop2.20}.

Thus, it remains to show that every $Q$ as in Case 2 satisfies the $\eps$-local WHSA property. 
In fact, we shall prove the following.   Given $\eps>0$, we set
$$\bqo=\bqo(\eps):= B\left(x_Q,\eps^{-8}\ell(Q)\right)\,,\qquad \dqo:= \bqo\cap\pom.$$
\begin{lemma}\label{LVlemma}  Fix $\eps\in (0,K_0^{-6})$.   Suppose that $u\geq 0$ is harmonic in
$\om_Q:= \om\cap\bqo$, $u\in C(\overline{\om_Q})$, $u\equiv 0$ on $\dqo$.  Suppose also that
for some $i$, there exists a point $Y_Q\in U_Q^i$ such that
\begin{equation}\label{eq4.17a}
|\nabla u(Y_Q)| \approx 1\,,
\end{equation}
and furthermore, that
\begin{equation}\label{eq4.18}
\sup_{B_Q^{***}} u \lesssim \eps^{-5} \ell(Q)\,;
\end{equation}
\begin{equation}\label{eq4.19}
\sup_{X,Y\in \tU^i_{Q}}\,\sup_{Z_1\in B_Y,\,Z_2\in B_X} |\nabla u(Z_1)
 -\nabla u(Z_2)|\,\leq\, \eps^{2M}\,.
\end{equation}
Then $Q$ satisfies the $\eps$-local WHSA, provided that $M$ is large enough, depending only on dimension
and on the implicit constants in the stated hypotheses.
\end{lemma}

Let us recall that in the scenario of Case 2,
$Q\in \dd_{\F,Q_0}$, $\ell(Q)\le \eps^{10}\,\ell(Q_0)$, and \eqref{eq5.2} holds.   Then 
\eqref{eq4.17a} holds by virtue of \eqref{eq6.9*}, while \eqref{eq4.18} holds
by Lemma \ref{lemma:G-aver} applied with $B=2B_Q^{***}$, and \eqref{eq4.9}.
Moreover, \eqref{eq4.19} is merely
a restatement of \eqref{eq5.2}.  Thus, the hypotheses of Lemma \ref{LVlemma} are all verified
for the Case 2 cubes, so modulo the proof of Proposition \ref{prop2.20}, it remains only to prove the lemma.

\begin{proof}[Proof of Lemma \ref{LVlemma}]
Our approach here will follow that of \cite{LV} very closely, but with some modifications owing to the
fact that in contrast to the situation in \cite{LV}, our solution $u$ need not be Lipschitz,
and our harmonic measures need not be doubling (it is the latter obstacle that has forced us to
introduce the WHSA condition, rather than to work with the ``Weak Exterior Convexity" condition
used in \cite{LV}).  In fact, Lemma \ref{LVlemma} is essentially a distillation of the main argument
of the corresponding part of \cite{LV}, but with the doubling hypothesis removed.

For convenience,
in the proof of the lemma, we shall
use the notational convention that implicit and generic constants are allowed to depend upon
$K_0$, but not on $\eps$ or $M$. 
Dependence on the latter will be noted explicitly.

We begin with the following.  We remind the reader that the balls $B_Y$ and $\tB_Y$ are defined in \eqref{eq5.1}.
\begin{lemma}\label{l5.14}  Let $Y\in U^i_Q$, $X\in \tU^i_Q$.  Suppose first that $w\in \partial\tB_Y\cap\pom$, and
let $W$ be the radial
projection of $w$ onto $\partial B_Y$.  Then
\begin{equation}\label{eq5.15a}
u(W)\, \lesssim \, \eps^{2M-5} \delta(Y)\,.
\end{equation}
If $w\in \partial\tB_X\cap\pom$, and $W$ now is the radial
projection of $w$ onto $\partial B_X$, then
\begin{equation}\label{eq5.16a}
u(W)
\lesssim
\eps^{2M-5} \ell(Q)\,.
\end{equation}
\end{lemma}

\begin{proof}  Since $K_0^{-1}\ell(Q)\lesssim \delta(Y)\lesssim K_0\, \ell(Q)$ for $Y\in U_Q^i$, it is enough to prove
\eqref{eq5.16a}.

To prove \eqref{eq5.16a}, we first
note that 
$$|W-w|\, =\, \eps^{2M/\alpha}\delta(X)\,\lesssim\, \eps^{2M/\alpha}\eps^{-3}\ell(Q)\,,$$ 
by definition of $B_X,\tB_X$ and the fact that by construction of $\tU^i_Q$,
\begin{equation}\label{eq5.27a}
\eps^3\ell(Q)\lesssim \delta(X)
\lesssim \eps^{-3}\ell(Q)\,,\quad \forall \, X\in\tU_Q^i\,.
\end{equation}
In addition,
\begin{equation}\label{eq5.29a}
\diam(\tU_Q^i)\lesssim \eps^{-4}\ell(Q)\,,
\end{equation}
again by construction of $\tU_Q^i$.  Consequently, $W\in (1/2) B_Q^{***}=B\big(x_Q, (1/2 \eps^{-5}\ell(Q)\big)$, 
so by Corollary \ref{cor2.12} and \eqref{eq4.18},
\begin{equation*}
u(W)\, \lesssim \,\left(\frac{\eps^{2M/\alpha}\eps^{-3}\ell(Q)}{\eps^{-5}\ell(Q)}\right)^\alpha \,
\frac{1}{|B_Q^{***}|} \int\!\!\!\int_{B_Q^{***}}\, u\,\lesssim \,
\eps^{2M+2\alpha -5}\ell(Q)\,\leq \eps^{2M -5}\ell(Q)\,.
\end{equation*}
\end{proof}

\begin{claim}\label{claim6.14} Let $Y\in U^i_Q$.
For all $W\in  B_Y$ (see \eqref{eq5.1}),
\begin{equation}\label{eq6.15}
|u(W) -u(Y) -\nabla u(Y) \cdot(W-Y)|\lesssim \eps^{2M} \delta(Y)\,.
\end{equation}
\end{claim}

\begin{proof}[Proof of Claim \ref{claim6.14}]
Let $W\in B_Y$.  Then
$$
u(W) -u(Y)
=
\nabla u (\widetilde{W})\cdot (W-Y)\,,
$$
for some $\widetilde{W}\in B_Y$.  We may then invoke \eqref{eq4.19}, with $X=Y$, $Z_1 = \widetilde{W}$, and $Z_2 = Y$,
to obtain \eqref{eq6.15}.
\end{proof}

\begin{claim}\label{claim6.16}  Let $Y\in U^i_Q$.  Suppose that  $w\in \partial\tB_Y\cap\pom$. 
Then
\begin{equation}\label{eq6.17}
|u(Y) -\nabla u(Y) \cdot(Y-w)|=|u(w)-u(Y) -\nabla u(Y) \cdot(w-Y)|
\lesssim
\eps^{2M-5} \delta(Y)\,.
\end{equation}
\end{claim}

\begin{proof}[Proof of Claim \ref{claim6.16}]
Given $w\in \tB_Y\cap\pom$, let $W$ be the radial
projection of $w$ onto $\partial B_Y$, so that $|W-w| = \eps^{2M/\alpha}\delta(Y)$.
Since $u(w)=0$, by \eqref{eq5.15a}
we have
$$
|u(W)-u(w)| = u(W)
\lesssim
\eps^{2M-5} \delta(Y).
$$
Since \eqref{eq6.15} holds
for $W$, we obtain \eqref{eq6.17}.
\end{proof}

To simplify notation, let us now set $Y:=Y_Q$, the point in $U_Q^i$ satisfying
\eqref{eq4.17a}. 
By \eqref{eq4.17a} and  \eqref{eq4.19}, for $\eps<1/2$, and $M$ chosen large enough, we
have that
\begin{equation}\label{eq6.5aa}
|\nabla u(Z)| \approx 1\,,\qquad \forall\, Z\in \tU^i_Q\,.
\end{equation}
By translation and rotation, we may assume that $0\in \tB_Y\cap\pom$, and that
$Y =\delta(Y) e_{n+1}$, where as usual $e_{n+1}:=(0,\dots,0,1)$.

\begin{claim}\label{claim6.19}  We claim that
\begin{equation}\label{eq6.20}
\big|\, \langle \nabla u(Y),e_{n+1}\rangle -|\nabla u(Y)|\,\big|\lesssim \eps^{2M-5}\,.
\end{equation}
\end{claim}
\begin{proof}[Proof of Claim \ref{claim6.19}]
We apply \eqref{eq6.17}, with $w=0$, to obtain
$$
|u(Y) -\nabla u(Y)\cdot Y|
\lesssim
\eps^{2M-5}\delta(Y).
$$
Combining the latter bound with \eqref{eq6.15}, we find that
\begin{multline}\label{eq6.21}
|u(W)  -\nabla u(Y)\cdot W|\,=\,
|u(W) -\nabla u(Y)\cdot Y -\nabla u(Y)\cdot (W-Y)|
\\[4pt]\lesssim\, \eps^{2M-5}\delta(Y)\,,\quad \forall\, W\in B_Y\,.
\end{multline}
Fix $W\in \partial B_Y$ so that
\begin{equation*}
\nabla u(Y)\cdot\frac{W-Y}{|W-Y|} = -|\nabla u(Y)|\,.
\end{equation*}
Since $|W-Y|= (1-\eps^{2M/\alpha})\delta(Y)$, and since $u\geq 0$, we have
\begin{multline}\label{eq6.22}
0\leq |\nabla u(Y)|-\nabla u(Y)\cdot e_{n+1}
\leq |\nabla u(Y)|-\nabla u(Y)\cdot e_{n+1} +\frac{u(W)}{\delta(Y)}\\[4pt]
\leq \frac1{\delta(Y)}\left(-\nabla u(Y) \cdot \frac{(W-Y)}{1-\eps^{2M/\alpha}}\, -\,\nabla u(Y)\cdot Y \,+\,u(W) \right)
\\[4pt]
\lesssim
\left(\eps^{2M-5} +\eps^{2M/\alpha}\right)
\approx
\eps^{2M-5}\,,
\end{multline}
by \eqref{eq6.21} and \eqref{eq4.17a}.
\end{proof}

\begin{claim}\label{claim6.24}  Suppose that $M>5$.  Then
\begin{equation}\label{eq6.25}
\big|\, |\nabla u(Y)|e_{n+1}-\nabla u(Y)\,\big|
\lesssim
\eps^{M-3}\,.
\end{equation}
\end{claim}
\begin{proof}[Proof of Claim \ref{claim6.24}]
Note that
$$
\big|\, |\nabla u(Y)|e_{n+1}-\langle \nabla u(Y),e_{n+1}\rangle e_{n+1}\,\big|\lesssim \,
\eps^{2M-5}\,,
$$
by Claim \ref{claim6.19}.  Therefore,
it is enough to consider
$\nabla_\| u:= \nabla u - \langle\nabla u,e_{n+1}\rangle e_{n+1}$.
Observe that
\begin{multline*}
|\nabla_\| u(Y)|^2
=
|\nabla u(Y)|^2 -\big(\langle\nabla u(Y),e_{n+1}\rangle\big)^2
\\[4pt]
=
\big(|\nabla u(Y)|-\langle\nabla u(Y),e_{n+1}\rangle\big)\,
\big(|\nabla u(Y)|+\langle\nabla u(Y),e_{n+1}\rangle\big)
\lesssim
\eps^{2M-5}\,,
\end{multline*}
by \eqref{eq6.20} and \eqref{eq4.17a}.
\end{proof}

Now for $Y=\delta(Y)e_{n+1}\in U_{Q}^i$ fixed as above, we consider another point
$X\in\tU_Q^i$.  
We form a polygonal path in $\tU^i_Q$, joining $Y$ to $X$, with vertices
$$Y_0:=Y, Y_1,Y_2,\dots,Y_N:= X\,,$$
such that 
$Y_{k+1}\in B_{Y_k}\cap B(Y_k,\ell(Q)),\, 0\leq k\leq N-1$, and such that
the distance between
consecutive vertices is comparable to $\ell(Q)$, and therefore the
total length of the path is on the order of $N\ell(Q)$.  Let us note that we may take
$N\lesssim \eps^{-4}$,  by \eqref{eq5.29a}.
We further note that by \eqref{eq4.19} and \eqref{eq6.25},
\begin{multline}\label{eq6.26}
\big|\nabla u(W) -|\nabla u(Y)|e_{n+1}\big|
\\[4pt]
\leq \,|\nabla u(W)-\nabla u(Y)| + \big|\nabla u(Y) -|\nabla u(Y)|e_{n+1}\big|
\\[4pt]
\lesssim
\eps^{2M}+\eps^{M-3}
\lesssim
\eps^{M-3}\,,\quad \forall\,W\in B_Z\,,\,\,\forall Z\in \tU^i_Q\,.
\end{multline}

\begin{claim}\label{claim6.27} Assume $M>7$.  Then for each $k=1,2,\dots,N$,
\begin{equation}\label{eq6.28}
\big|u(Y_{k})-|\nabla u(Y)|\langle Y_k,e_{n+1}\rangle\big| \,\lesssim\, k\,\eps^{M-3}\ell(Q)\,.
\end{equation}
Moreover,
\begin{equation}\label{eq6.28a}
\big|u(W)-|\nabla u(Y)|W_{n+1}\big| \,\lesssim\,\eps^{M-7}\ell(Q)\,,\quad \forall \, W\in B_{X}\,,\, \forall \,X\in\tU^i_Q\,.
\end{equation}
\end{claim}

\begin{proof}[Proof of Claim \ref{claim6.27}]
By \eqref{eq6.21} and \eqref{eq6.25}, we have
\begin{multline}\label{eq6.29}
\big| u(W) -|\nabla u(Y)|W_{n+1}\big|\,\lesssim\,|u(W)  -\nabla u(Y)\cdot W|\,+\,
\big|\Big(\nabla u(Y) -|\nabla u(Y)|e_{n+1}\Big)\cdot W\big|\\[4pt]
\lesssim\, \eps^{2M-5}\delta(Y) +\eps^{M-3} |W|
\,\lesssim \, \eps^{M-3}\ell(Q) \,, \quad \forall\,W\in B_Y\,,
\end{multline}
since $\delta(Z) \approx \ell(Q),$ for all $Z\in U^i_Q$ (so in particular, for $Z=Y$), and since
$|W|\leq 2\delta(Y) \lesssim \ell(Q)$, for all $W\in B_Y$.
Thus, \eqref{eq6.28} holds with $k=1$, since $Y_1\in B_Y$, by construction.

Now suppose that \eqref{eq6.28} holds for
all $1\leq i\leq k$, with $k\leq N$.   
Let $W\in B_{Y_k}$, so that
$W$ may be joined to $Y_k$ by a line segment of length less than $\delta(Y_{k})
\lesssim \eps^{-3}\ell(Q)$ (the latter bound holds by \eqref{eq5.27a}).
We note also  that if $k\leq N-1$, and
if $W=Y_{k+1}$, then this line segment has length at most $\ell(Q)$, by construction.
Then
\begin{multline*}
\big|u(W)-|\nabla u(Y)|W_{n+1}\big| \\[4pt]
\leq\,|u(W)-u(Y_{k})
+|\nabla u(Y)|\langle (Y_{k}-W),e_{n+1}\rangle\big|\,+\,
\big|u(Y_{k})  -|\nabla u(Y)|\langle Y_{k},e_{n+1}\rangle\big|\\[4pt]
= \big|(W-Y_{k})\cdot \nabla u(W_1) +|\nabla u(Y)|\langle (Y_{k}-W),e_{n+1}\rangle\big|
\,+\, O\left(k\,\eps^{M-3} \ell(Q)\right)\,,
\end{multline*}
where $W_1$ is an appropriate point on the line segment joining $W$ and $Y_k$, and
where we have used that $Y_{k}$ satisfies \eqref{eq6.28}.
By \eqref{eq6.26}, applied to $W_1$, we find in turn that
\begin{equation}\label{eq6.30}
\big|u(W)-|\nabla u(Y)|W_{n+1}\big|
\lesssim
\eps^{M-3}\, |W-Y_k|
+
k\, \eps^{M-3} \ell(Q)\,,
\end{equation}
which, by our previous observations, is bounded by $C(k+1)\eps^{M-3} \ell(Q)$, if $W=Y_{k+1}$,
or by $(\eps^{M-6} +\,k\,\eps^{M-3} )\ell(Q),$ in general.  In the former case, we find that
\eqref{eq6.28} holds for all $k=1,2,\dots,N$, and in the latter case, taking $k=N\lesssim
\eps^{-4}$, we obtain \eqref{eq6.28a}.
\bigskip
\end{proof}

\begin{claim}\label{claim6.32} Let $X\in \tU^i_Q$, and
let $w\in\pom\cap \partial\tB_X$.  Then
\begin{equation}\label{eq6.33}
|\nabla u(Y)|\, |w_{n+1}|
\lesssim
\eps^{M/2}\ell(Q)\,.
\end{equation}
\end{claim}

\begin{proof}[Proof of Claim \ref{claim6.32}]
Let $W$ be the radial projection of $w$ onto $\partial B_X$, so that
\begin{equation}\label{eq6.34}
|W-w| = \eps^{2M/\alpha}\delta(X) \lesssim \eps^{(2M/\alpha) -3}\ell(Q)
\,,\end{equation}  
by \eqref{eq5.27a}.
We write
\begin{multline*}
|\nabla u(Y)|\, |w_{n+1}| 
\,\leq \,
 |\nabla u(Y)|\,|W-w|\,+\,
\big|u(W) -|\nabla u(Y)|W_{n+1}\big| \,+\,
u(W) 
 \\[4pt]
=:I+II+u(W).
\end{multline*}
Note that $I\lesssim \eps^{(2M/\alpha)-3}\ell(Q)$, by \eqref{eq6.34} and \eqref{eq4.17a} (recall that $Y=Y_Q$),
and that $II \lesssim \eps^{M-7}\ell(Q)$, by \eqref{eq6.28a}.
Furthermore, $u(W) \lesssim \eps^{2M-5} \ell(Q)$, by \eqref{eq5.16a}.
For $M$ chosen large enough, we obtain \eqref{eq6.33}.
\end{proof}

We note that since we have fixed $Y=Y_Q$, 
 it then follows from \eqref{eq6.33} and \eqref{eq4.17a} that
\begin{equation}\label{eq5.35}
|w_{n+1}|\,\lesssim \,\eps^{M/2} \ell(Q)\,,\quad \forall\, w\in\pom\cap\partial\tB_X\,,\quad \forall X\in \tU^i_Q\,.
\end{equation}

Recall now that $x_Q$ denotes the ``center" of $Q$ (see \eqref{cube-ball}-\eqref{cube-ball2}).
Set
\begin{equation}\label{eq5.41a}
\mathcal{O}:= B\left(x_Q,2\eps^{-2}\ell(Q)\right) \cap \left\{W:W_{n+1}> \eps^2 \ell(Q)\right\}\,.
\end{equation}

\begin{claim}\label{claim5.40} For every point  $X\in \oo$, we have
$X\approx_{\eps,Q} Y$ (see Definition \ref{def2.11a}).
Thus, in particular, $\oo\subset \tU^i_Q.$
\end{claim}
\begin{proof}[Proof of Claim \ref{claim5.40}]  Let $X\in \oo$.
We need to show that $X$ may be connected to $Y$ by a chain of
at most $\eps^{-1}$ balls of the form $B(Y_k,\delta(Y_k)/2)$, with
$\eps^3\ell(Q)\leq\delta(Y_k)\leq \eps^{-3}\ell(Q)$ (for convenience, we shall refer to such balls as
``admissible").   We first observe that
if $X = te_{n+1}$, with $\eps^{3} \ell(Q) \leq t\leq \eps^{-3}\ell(Q)$, then by an iteration argument using \eqref{eq5.35}
(with $M$ chosen large enough),
we may join $X$ to $Y$ by at most $C\log(1/\eps)$ admissible balls.  The
point $(2\eps)^{-3} \ell(Q) e_{n+1}$ may then be joined to any point of the form $(X', (2\eps)^{-3} \ell(Q))$
by a chain of at most $C$ admissible balls, whenever
$X'\in \rn$ with $|X'|\leq \eps^{-3}\ell(Q)$.  In turn,  the latter point may then be joined
to $(X', \eps^{3} \ell(Q))$.
\end{proof}

We note that Claim \ref{claim5.40} implies that
\begin{equation}\label{eq5.41}
\pom\cap \oo =\emptyset\,.
\end{equation}
Indeed, $\oo\subset \tU^i_Q\subset \Omega$.

Let $P_0$ denote the hyperplane
$$P_0:=\{Z:\,Z_{n+1} = 0\}\,.$$

\begin{claim}\label{claim5.42}
If $Z\in P_0$, with
$|Z-x_Q|\leq \eps^{-2}\ell(Q)$, then
\begin{equation}\label{eq5.43}
\delta(Z)=\dist(Z,\pom) \leq 16 \eps^2\ell(Q)\,.
\end{equation}
\end{claim}
\begin{proof}[Proof of Claim \ref{claim5.40}]
Observe that $B(Z, 2 \eps^2\ell(Q))$ meets $\oo$.  Then by Claim
\ref{claim5.40}, there is a point $X\in \tU^i_Q \cap B(Z, 2 \eps^2\ell(Q))$.
Suppose now that \eqref{eq5.43} is false, so that $\delta(X) \geq14\eps^2\ell(Q).$
Then $B(Z, 4 \eps^2\ell(Q))\subset B_X$, so by \eqref{eq6.28a}, we have
\begin{equation}\label{eq5.44}
\left|u(W)-|\nabla u(Y)|W_{n+1}\right|
\leq
C\, \eps^{M-7}\ell(Q)\,,\quad \forall \, W\in
B(Z, 4 \eps^2\ell(Q))\,.
\end{equation}
In particular, since $Z_{n+1}=0$, we may choose $W$ such that $W_{n+1} = -\eps^2\ell(Q)$,
to obtain that
$$|\nabla u(Y)|\, \eps^2\ell(Q) 
\,\leq\, C \eps^{M-7}\ell(Q)\,,$$
since $u\geq 0$.  But for $\eps<1/2$, and $M$ large enough, 
this is a contradiction, by
\eqref{eq4.17a} (recall that we have fixed $Y=Y_Q$).
\end{proof}

It now follows by Definition \ref{def2.13} that $Q$ satisfies the $\eps$-local WHSA
condition, with
$$P=P(Q):= \{Z: \, Z_{n+1}= \eps^2\ell(Q)\}\,,\quad H=H(Q):= \{Z: \, Z_{n+1}>\eps^2\ell(Q)\} \,.$$
This concludes the proof of Lemma \ref{LVlemma}, and therefore also that of
Theorem \ref{t1}, modulo Proposition \ref{prop2.20}.
\end{proof}

\section{WHSA implies UR:  Proof of Proposition \ref{prop2.20}}\label{s7}
We suppose that $E$ satisfies the WHSA property.  Given a positive $\eps<\eps_0\ll K_0^{-6}$, we
let $\B_0$ denote the collection of bad cubes for which $\eps$-local WHSA fails.
By Definition \ref{def2.14}, $\B_0$ satisfies the Carleson packing condition
\eqref{eq2.pack2}.  We now introduce a variant of the packing measure for
$\B_0$.
We recall that $B^*_Q= B(x_Q, K_0^2\ell(Q)),$ and given $Q\in\dd(E)$, we set
\begin{equation}\label{eq5.8aa}
\dd_\eps(Q):=\left\{Q'\in\dd(E):\, \eps^{3/2}\ell(Q)\leq \ell(Q')\leq \ell(Q),\,Q'\,{\rm meets}\,\, B_Q^*
\right\}\,.
\end{equation}
Set
\begin{equation}\label{eq4.0}
\alpha_Q:=\left\{
\begin{array}{ll}
\sigma(Q)\,,&{\rm if}\,\, \B_0\cap \dd_\eps(Q)\neq \emptyset,
\\[6pt]
0\,,&{\rm otherwise},
\end{array}
\right.
\end{equation}
and define
\begin{equation}\label{eq4.1}
\mut(\dd'):= \sum_{Q\in\dd'}\alpha_Q\,,\qquad \dd'\subset\dd(E)\,.
\end{equation}
Then $\mut$ is a discrete Carleson measure, with
\begin{equation}\label{eq6.2}
\mut(\dd_{Q_0})\, =\sum_{Q\subset Q_0}\alpha_Q \leq\, C_\eps\,\sigma(Q_0)\,,\qquad Q_0\in \dd(E)\,.
\end{equation}
Indeed, note that
for any $Q'$, the cardinality of $\{Q:\, Q'\in \dd_\eps(Q)\}$, is uniformly bounded, depending on $n$,
$\eps$ and $ADR$, and that $\sigma(Q)\leq C_\eps \sigma(Q')$, if $Q'\in\dd_\eps(Q)$.   Then
given any
$Q_0\in\dd(E)$,
\begin{multline*}
\mut(\dd_{Q_0})\, =\sum_{Q\subset Q_0:\,\B_0\cap\dd_\eps(Q)\neq\emptyset}\sigma(Q)\,
\leq \sum_{Q'\in \B_0}\,\sum_{Q\subset Q_0:\, Q'\in \dd_\eps(Q)}\sigma(Q) \\[4pt]
\leq \, C_\eps\!\!\sum_{Q'\in \B_0:\, Q' \subset 2B_{Q_0}^*}\sigma(Q')\,\leq\,\C_\eps\, \sigma(Q_0)\,,
\end{multline*}
by \eqref{eq2.pack2} and ADR.

To prove Proposition \ref{prop2.20}, we are required to show that
the collection $\B$ of bad cubes for which the $\sqrt{\eps}$-local BAUP condition fails, satisfies a packing
condition.  That is, we shall establish the discrete Carleson measure estimate
\begin{equation}\label{eq6.3}
\mutt(\dd_{Q_0})\, =\sum_{Q\subset Q_0:\, Q\in\B}\sigma(Q) \leq\, C_\eps\,\sigma(Q_0)\,,\qquad Q_0\in \dd(E)\,.
\end{equation}
To this end,  by \eqref{eq6.2}, it suffices to show that if 
$Q\in \B$, then $\alpha_Q\neq 0$ (and thus $\alpha_Q=\sigma(Q)$, by definition).
In fact, we shall prove the contrapositive statement.
\begin{claim}\label{claim6.15} Suppose then that $\alpha_Q = 0$. 
Then $\sqrt{\eps}$-local BAUP condition holds for $Q$.
\end{claim}

\begin{proof}[Proof of Claim \ref{claim6.15}]  
We first note that since $\alpha_Q=0$, then
by definition of $\alpha_Q$,
\begin{equation}\label{eq6.14}
\B_0\cap \dd_\eps(Q) =\emptyset\,. 
\end{equation}
Thus, the $\eps$-local WHSA condition (Definition \ref{def2.13})
holds for every $Q'\in \dd_\eps(Q)$ (in particular,
for $Q$ itself).  By rotation and translation, we may suppose that the hyperplane $P= P(Q)$ in Definition
\ref{def2.13} is
$$P = \left\{Z\in \ree:\, Z_{n+1} = 0\right\}\,,$$
and that the half-space $H=H(Q)$ is
the upper half-space $\reu = \{Z:\, Z_{n+1}>0\}$.
We recall that by Definition \ref{def2.13}, $P$ and $H$ satisfy
\begin{equation}\label{eq6.16}
\dist(Z,E)\leq\eps\ell(Q)\,, \qquad \forall \, Z\in P\cap B_Q^{**}(\eps)\,.
\end{equation}
\begin{equation}\label{eq6.16a}
\dist(P,Q)\leq K_0^{3/2}\ell(Q)\,,
\end{equation}
and
\begin{equation}\label{eq6.17a}
H\cap B_Q^{**}(\eps)\cap E=\emptyset\,.
\end{equation}

The proof will now follow a similar construction in \cite{LV},
which was used to establish the Weak Exterior Convexity condition.
By \eqref{eq6.17a}, there are two cases.

\smallskip

\noindent {\bf Case 1}:  $10Q \subset \{ Z:\, -\sqrt{\eps}\ell(Q) \leq Z_{n+1} \leq 0\}$.
In this case, the $\sqrt{\eps}$-local BAUP condition holds trivially for $Q$, with $\P =\{P\}$.

\smallskip

\noindent{\bf Case 2}.  There is a point $x\in 10Q$ such that $x_{n+1}<-\sqrt{\eps}\ell(Q)$.
In this case, we choose $Q'\ni x$, with $\eps^{3/4}\ell(Q)\leq \ell(Q') < 2\eps^{3/4}\ell(Q).$
Thus,
\begin{equation}\label{eq6.19aa}
Q'\subset \big\{Z:\, Z_{n+1}\leq -\frac12\sqrt{\eps}\ell(Q)\big\}\,.
\end{equation}
Moreover, $Q'\in\dd_\eps(Q)$, so by \eqref{eq6.14}, $Q'\notin \B_0$, i.e.,
$Q'$ satisfies the $\eps$-local WHSA.  Let $P'=P(Q')$, and $H'=H(Q')$ denote the
hyperplane and half-space corresponding to $Q'$ in Definition \ref{def2.13}, so that
\begin{equation}\label{eq6.18}
\dist(Z,E)\leq\eps\ell(Q')\leq 2 \eps^{7/4}\ell(Q)\,, \quad \forall \, Z\in P'\cap B_{Q'}^{**}(\eps)\,,
\end{equation}
\begin{equation}\label{eq6.18a}
\dist(P',Q')\leq K_0^{3/2}\ell(Q')\approx K_0^{3/2} \eps^{3/4}\ell(Q) \ll\eps^{1/2}\ell(Q)
\end{equation}
(where the last inequality holds since $\eps \ll K_0^{-6}\,$), and
\begin{equation}\label{eq6.19}
H'\cap B_{Q'}^{**}(\eps)\cap E=\emptyset\,,
\end{equation}
where we recall that $B_{Q'}^{**}(\eps) := B\left(x_{Q'},\eps^{-2}\ell(Q')\right)$
(see \eqref{eq2.bstarstar}).
We note that
\begin{equation}\label{eq6.22a}
B_Q^*\subset \tB_Q(\eps):= B\left(x_Q, \eps^{-1}\ell(Q)\right) \subset B_{Q'}^{**}(\eps)\cap B_Q^{**}(\eps)\,,
\end{equation}
by construction, since $\eps\ll K_0^{-6}$. 
Let $\nu'$ denote the unit normal vector to $P'$, pointing into $H'$.  Then
$\nu'$ points ``downward", i.e., 
$\nu'\cdot e_{n+1} <0$, otherwise $H'\cap \tB_Q(\eps)$ would meet $E$, by \eqref{eq6.16},
\eqref{eq6.19aa}, and \eqref{eq6.18a}. 
Moreover,
by  \eqref{eq6.17a}, \eqref{eq6.18}, and the definition of $H$,
\begin{equation}\label{eqn:above}
P'\cap \tB_Q(\eps)\cap\{Z:Z_{n+1} >2\eps^{7/4}\,\ell(Q)\}=\emptyset\,.
\end{equation}
Consequently,
\begin{claim}\label{claim:nu'-down}
The angle $\theta$
between $\nu'$ and $-e_{n+1}$ satisfies $\theta \approx \sin\theta \lesssim \eps$.
\end{claim}
Indeed, since $Q'$ meets $10Q$, 
 \eqref{eq6.16a} and \eqref{eq6.18a} imply that
 $\dist(P,P')\lesssim K_0^{3/2} \ell(Q)$, and that the latter estimate is attained near $Q$.
By \eqref{eqn:above} and a trigonometric argument, 
one then obtains Claim \ref{claim:nu'-down} 
(more precisely, one obtains 
$\theta \lesssim K_0^{3/2} \eps$, but in this section,
we continue to use the notational convention that 
implicit constants may depend upon $K_0$, but $K_0$ is fixed, and  $\eps \ll K_0^{-6}\,$).
The interested reader could probably supply the remaining details of the argument that we have just sketched, 
but for the sake of completeness, we shall give the full proof at the end of this section.

We therefore take Claim  \ref{claim:nu'-down} for granted, 
and proceed with the argument.  We note 
first that every point in $(P\cup P')\cap B_Q^*$ is at a distance at most $\eps\ell(Q)$ from $E$,
by \eqref{eq6.16}, \eqref{eq6.18} and \eqref{eq6.22a}.
To complete the proof of Claim \ref{claim6.15}, it therefore remains
only to verify the following.  As with the previous claim, we shall provide a condensed 
proof immediately, and present a more detailed argument 
at the end of the section.
\begin{claim}\label{claim5.18}
Every point in $10Q$ lies within $\sqrt{\eps}\ell(Q)$
of a point in $P\cup P'$.  
\end{claim} 
Suppose not.   We could then repeat the previous argument,
to construct a cube $Q''$,
a hyperplane $P''$, a unit vector $\nu''$ forming a small angle with $-e_{n+1}$,
and a half-space $H''$ with boundary $P''$, 
with the same properties as  $Q',P', \nu'$ and $H'$.  In particular, we have the respective
analogues of \eqref{eq6.18a}, \eqref{eq6.19aa}, and \eqref{eq6.19},
namely
\begin{equation}\label{eq6.23}
\dist(P'',Q'')\leq K_0^{3/2}\ell(Q')\approx K_0^{3/2} \eps^{3/4}\ell(Q) \ll\eps^{1/2}\ell(Q)\,,
\end{equation}
\begin{equation}\label{eq6.24}
\dist(Q'',P') \geq \frac12 \sqrt{\eps}\ell(Q)\,,\quad {\rm and}\quad Q''\cap H'=\emptyset\,,
\end{equation}
and
\begin{equation}\label{eq6.25a}
H''\cap B_{Q''}^{**}(\eps)\cap E=\emptyset\,,
\end{equation}
In addition, as in \eqref{eq6.22a}, we also have $B_Q^*\subset B_{Q''}^{**}(\eps)$.
On the other hand, the angle between $\nu'$ and $\nu''$
is very small.   Thus, 
combining \eqref{eq6.18}, \eqref{eq6.23} and \eqref{eq6.24},
we see that  $H''\cap B_Q^*$ captures points in $E$, which contradicts
\eqref{eq6.25a}.

Claim \eqref{claim6.15} therefore holds (in fact, with a union of at most 2 planes), 
and thus we obtain the conclusion of Proposition \ref{prop2.20}.
 \end{proof}

We now provide detailed proofs of Claims \ref{claim:nu'-down} and \ref{claim5.18}.

\begin{proof}[Proof  of Claim \ref{claim:nu'-down}]
By \eqref{eq6.18a} we can pick $x'\in Q'$, $y'\in P'$ such that $|y'-x'|\ll \eps^{1/2}\,\ell(Q)$ and therefore $y'\in 11\,Q$. Also, from \eqref{eq6.16a} and \eqref{eq6.17a} we can find $\bar{x}\in Q$ such that $-K_0^{3/2}\,\ell(Q)<\bar{x}_{n+1}\le 0$. This and \eqref{eq6.19aa} yield
\begin{equation}\label{loc-y'}
-2\,K_0^{3/2}\,\ell(Q)<y'_{n+1}<-\frac14\,\sqrt{\eps}\,\ell(Q)
\end{equation}

Write $\pi$ to denote the orthogonal projection onto $P$. Let $Z\in P$ (i.e., $Z_{n+1}=0$) be such that $|Z-\pi(y')|\le K_0^{3/2}\,\ell(Q)$. Then, $Z\in B(x_Q,3\,K_0^{3/2}\,\ell(Q))\subset B_Q^*$. Hence $Z\in P\cap B_Q^{**}(\eps)$ and by \eqref{eq6.16}, $\dist(Z,E)\le\eps\,\ell(Q)$. Then there exists $x_Z\in E$ with $|Z-x_Z|\le \eps\,\ell(Q)$ which in turn implies that $|(x_Z)_{n+1}|\le \eps\,\ell(Q)$. Note that $x_Z\in B(x_Q,4\,K_0^{3/2}\,\ell(Q))\subset B_Q^*$ and by \eqref{eq6.22a} $x_Z\in E\cap \cap B_Q^{**}(\eps)\cap B_{Q'}^{**}(\eps)$. This,  \eqref{eq6.17a} and \eqref{eq6.19} imply that $x_Z\not\in H\cup H'$. Hence, $(x_Z)_{n+1}\le 0$ and  $(x_Z-y')\cdot\nu'\le 0$, since $y'\in P'$ and $\nu'$ denote the unit normal vector to $P'$ pointing into $H'$. Observe that by
by \eqref{loc-y'}
\begin{equation}\label{x-y'}
\frac18\,\sqrt{\eps}\,\ell(Q)
<
-\eps\,\ell(Q)
+
\frac14\,\sqrt{\eps}\,\ell(Q)
<
(x_Z-y')_{n+1}
<2\,K_0^{3/2}\,\ell(Q)
\end{equation}
and
\begin{multline}\label{inner-prod}
(x_Z-y')_{n+1}\,\nu'_{n+1}
\le
-\pi(x_Z-y')\cdot \pi(\nu')
\\
\le
|x_Z-z|-\pi(Z-y')\cdot \pi(\nu')
\le
\eps\,\ell(Q)-\pi(Z-y')\cdot \pi(\nu')
\end{multline}
and that,

We prove that $\nu'_{n+1}<-\frac18<0$ considering two cases:

\noindent {\bf Case 1:} $|\pi(\nu')|\ge \frac12$.

We pick
$$
Z_1=\pi(y')+K_0^{3/2}\,\ell(Q)\,\frac{\pi(\nu')}{|\pi(\nu')|}.
$$
By construction $Z_1\in P$ and  $|Z_1-\pi(y')|\le K_0^{3/2}\,\ell(Q)$. Hence we can use \eqref{inner-prod}
\begin{multline*}
(x_{Z_1}-y')_{n+1}\,\nu'_{n+1}
\le
\eps\,\ell(Q)-\pi(Z_1-y')\cdot \pi(\nu')
\\
=
\eps\,\ell(Q)-K_0^{3/2}\,\ell(Q)\,|\pi(\nu')|
\le
-\frac14\,K_0^{3/2}\,\ell(Q).
\end{multline*}
This together with \eqref{x-y'} give that $\nu'_{n+1}<-1/8<0$.

\smallskip

\noindent {\bf Case 2:} $|\pi(\nu')|< \frac12$.

This case is much simpler. Note first that $|\nu'_{n+1}|^2=1-|\pi(\nu')|^2>3/4$ and thus either $\nu'_{n+1}<-\sqrt{3}/2$ or $\nu'_{n+1}>\sqrt{3}/2$. We see that the second scenario leads to a contradiction. Assume then that $\nu'_{n+1}>\sqrt{3}/2$. We take $Z_2=\pi(y')\in P$ which clearly satisfies  and  $|Z_2-\pi(y')|\le K_0^{3/2}\,\ell(Q)$. Again \eqref{inner-prod} and \eqref{x-y'} are applicable
$$
\frac18\,\sqrt{\eps}\,\ell(Q)\,\frac{\sqrt{3}}{2}
<
(x_Z-y')_{n+1}\,\nu_{n+1}'
\le
\eps\,\ell(Q)
\ll
\sqrt{\eps}\,\ell(Q),
$$
and we get a contradiction. Hence necessarily $\nu'_{n+1}\le -\sqrt{3}/{2}<-1/8<0$.

\smallskip

Once we know that  $\nu'_{n+1}<-1/8<0$ we estimate $\theta$ the angle between $\nu'$ and $-e_{n+1}$. Note first $\cos\theta=-\nu'_{n+1}>1/8$. If $\cos\theta=1$ (which occurs if $\nu'=-e_{n+1}$) then $\theta=\sin\theta=0$ and the proof is complete. Assume then that $\cos\theta\neq 1$ in which case $1/8<-\nu'_{n+1}<1$ and hence $|\pi(\nu')|\neq 0$. Pick
$$
Z_3
=
y'+\frac{\ell(Q)}{2\,\eps}\,
\hat{\nu}',
\qquad\qquad
\hat{\nu}'=
\frac{e_{n+1}-\nu'_{n+1}\,\nu'}{|\pi(\nu')|}.
$$
Note that $\hat{\nu}'\dot\nu'=0$ and then $Z_3\in P'$ since $y'\in P'$. Also $|\hat{\nu}'|=1$ and therefore $|Z_3-y'|=\ell(Q)/(2\,\eps)$. This in turn gives that $Z_3\in \tB_Q(\eps)$. We have then obtained that $Z_3\in P'\cap\tB_Q(\eps)$ and hence $(Z_3)_{n+1}\le 2\eps^{7/4}\,\ell(Q)$ by \eqref{eqn:above}. This and \eqref{x-y'} easily give
\begin{multline*}
4\,K_0^{3/2}
\,\ell(Q)
\ge
2\eps^{7/4}\,\ell(Q)
\ge
(Z_3)_{n+1}
=
y'_{n+1}
+
\frac{\ell(Q)}{2\,\eps}\,\frac{1-(\nu'_{n+1})^2}{|\pi(\nu')|}
\\
=
y'_{n+1}
+
\frac{\ell(Q)}{2\,\eps}\,|\pi(\nu')|
\ge
-2\,K_0^{3/2}\,\ell(Q)+
\frac{\ell(Q)}{2\,\eps}\,|\pi(\nu')|.
\end{multline*}
This readily yields $|\sin \theta|=|\pi(\nu')|\le 8\,K_0^{3/2}\,\eps$ and the proof is then complete.
\end{proof}

\begin{proof}[Proof  of Claim \ref{claim5.18}]
We seek to show that every point in $10Q$ 
lies within $\sqrt{\eps}\ell(Q)$
of a point in $P\cup P'$.  
Suppose not and we will get a contradiction.

Assume then that there is $x'\in 10Q$ with $\dist(x',P\cup P')>\eps \,\ell(Q)$. In particular Then $x'_{n+1}<-\sqrt{\eps}\,\ell(Q)$
and we  repeat the previous argument,
to construct a cube $Q''$, a hyperplane $P''$, a unit vector $\nu''$ forming a small angle with $-e_{n+1}$,
and a half-space $H''$ with boundary $P''$,
with the same properties as  $Q',P', \nu'$ and $H'$.  In particular, we have the respective
analogues of \eqref{eq6.18a} and \eqref{eq6.19},
namely
\begin{equation}\label{eq6.23}
\dist(P'',Q'')\leq K_0^{3/2}\ell(Q')\approx K_0^{3/2} \eps^{3/4}\ell(Q) \ll\eps^{1/2}\ell(Q)\,,
\end{equation}
and
\begin{equation}\label{eq6.25a}
H''\cap B_{Q''}^{**}(\eps)\cap E=\emptyset\,,
\end{equation}
Also, 
$$
\eps\ell(Q)
\le
\dist(x',P')
\le
\diam(Q'')+\dist(Q'', P')
\le
\frac12\,\eps\,\ell(Q)+\dist(Q'', P'),
$$
and, by \eqref{eq6.19},
\begin{equation}\label{eq6.24}
\dist(Q'',P') \geq \frac12 \sqrt{\eps}\ell(Q)\,,\quad {\rm and}\quad Q''\cap H'=\emptyset\,,
\end{equation}
In addition, as in \eqref{eq6.22a}, we also have $B_Q^*\subset B_{Q''}^{**}(\eps)$.

By \eqref{eq6.23} there is $y''\in Q''$ and $z'' \in P''$ such that $|y''-z''|\ll \eps^{1/2}\,\ell(Q)$. By \eqref{eq6.25a} $y''\not\in H'$. Write $\pi'$ to denote the orthogonal projection onto $P'$ and note that \eqref{eq6.24} give $\dist(y'',P')=|y''-\pi'(y'')|\ge \frac12 \sqrt{\eps}\ell(Q)$. Note also that
$$
|y''-\pi'(y'')|
=\dist (y'', P')
\le
|y''-x'|+|x'-x|+\diam(Q')+dist(Q',P')
\le 11\diam(Q)
$$
and
$$
|\pi'(y'')-x_Q|
\le
|\pi'(y'')-y''|+|y''-x'|+|x'-x_Q|
<22\diam(Q)
<K^2\,\ell(Q).
$$
Hence $\pi'(y'')\in B_Q^*\subset\tB_Q(\eps)$ and since $\pi'(y'')\in P'$ we have that \eqref{eq6.18} give that there is $\tilde{y}\in E$ with $|\pi'(y'')-\tilde{y}|\le 2\,\eps^{7/4}\,\ell(Q)$. Then $\tilde{y}\in 23 Q\subset B_Q^*\cap E$ and $|\tilde{y}-z''|<12\,\diam(Q)$. To complete our proof we just need to show that $\tilde{y}\in H''$ which gets intro contradiction with \eqref{eq6.25a}.

Let us then proof that $\tilde{y}\in H''$. Write $\nu''$ to denote the unit normal vector to $P''$, pointing into $H''$ and let us claim that 
\begin{equation}\label{nu'-nu''}
|\nu'-\nu''|\le 16\,\sqrt{2}\,K_0^{2/3}\,\eps.
\end{equation}
Assuming this momentarily  and recalling that $y''\not\in H'$ we obtain. Then, 
\begin{multline*}
\frac12\,\sqrt{\eps}\,\ell(Q)
\le
|y''-\pi'(y'')|
=
(\pi(y'')-y'')\cdot\nu'
\\
\le
|\pi(y'')-\tilde{y}|
+
|\tilde{y}-z''|\,|\nu'-\nu''|+(\tilde{y}-z'')\cdot\nu''
+
|z''-y''|
\\
< 
\frac14\,\sqrt{\eps}\,\ell(Q)
+
(\tilde{y}-z'')\cdot\nu''.
\end{multline*}
This immediately gives that $(\tilde{y}-z'')\cdot\nu''>\frac14\,\sqrt{\eps}\,\ell(Q)>0$ and hence $y''\in H''$ as desired.

To complete the proof we obtain \eqref{nu'-nu''}. Note first that if $|\alpha|<\pi/4$ then
$$
1-\cos\alpha
=
1-\sqrt{1-\sin^2 \alpha}
\le
\sin^2\alpha.
$$
In particular we can apply this to $\theta$ (resp. $\theta'$) which is the angle between $\nu'$ (resp. $\nu''$) and $-e_{n+1}$
since $|\sin \theta|$, $|\sin\theta'|\le 8\,K_0^{3/2}\,\eps$:
$$
\sqrt{1-\cos\theta}
+
\sqrt{1-\cos\theta'}
\le
16\,K_0^{3/2}\,\eps
$$
Using the trivial formula
$$
|a-b|^2=2(1-a\dot b),
\qquad
\forall,\ a,b\in\ree,\ |a|=|b|=1.
$$
we conclude that
\begin{multline*}
|\nu'-\nu''|
\le
|\nu'-(-e_{n+1})|+|-(-e_{n+1})-\nu''|
=
\sqrt{2\,(1+\nu'\,e_{n+1})}
+
\sqrt{2\,(1+\nu''\,e_{n+1})}
\\[4pt]
=
\sqrt{2\,(1-\cos\theta)}
+
\sqrt{2\,(1-\cos\theta')}
\le
16\,\sqrt{2}\,K_0^{3/2}\,\eps.
\end{multline*}
This proves \eqref{nu'-nu''} and this completes the proof. 
\end{proof}

\section{Proof of Corollary \ref{c1}}\label{s-cor}

This result will follow almost immediately from Theorem \ref{t1}. Let $B=B(x, r)$ and
$\Delta=B\cap E$, with $x\in E$ and $0<r<\diam(E)$.  Let  $c$ be the constant
in Lemma \ref{Bourgainhm}.  By the ADR property of $E$, there is a point  $Y_{\Delta}\in B':= B(x, c\,r)$, which is
a Corkscrew point relative to the surface ball $ \Delta':= B'\cap E$, and therefore also
a Corkscrew point relative to $\Delta$, albeit with slightly different
Corkscrew constants now depending also on $c$, that is, there is $c_1$ such that $\delta(Y_\Delta)\ge c_1\,r$.  Note that $Y_\Delta$ satisfies the conditions in Theorem \ref{t1} and in order to apply that result we need to check the validity of $(a)$ and $(b)$. That $(a)$ holds is an immediate consequence of Lemma  \ref{Bourgainhm}. Let us now prove $(b)$. Given $C_1$ a large enough constant, write $\hat{\Delta}=C_1\Delta$. Cover $\hat{\Delta}$ by a collection of surface balls $\{\Delta_i\}_{i=1}^N$ with $\Delta_i=B_i\cap E=B(x_i,c_1\,r/4)$, $x_i\in \Delta$ and where $N$ is uniformly bounded. By construction $Y_\Delta\in\Omega\setminus 4\,B_i$ and by hypothesis $\hm^{Y_\Delta}\in$  weak-$A_\infty(2\,\Delta_i)$. Hence $\hm^{Y_\Delta}\ll\sigma$ in $2\Delta_i$ and \eqref{eqn:main-weak-RHP} holds with $Y_\Delta$ in place of $Y$ and with $\Delta'=\Delta_i$. Then, we clearly have that $\hm^{Y_\Delta}\ll\sigma$ in $\hat{\Delta}$ and if we write $k^{Y_\Delta}=d\hm^{Y_\Delta}/d\sigma$ we obtain
\begin{multline*}
\int_{\hat{\Delta}} k^{Y_\Delta}(z)^p\, d\sigma(z)
\le
\sum_{i=1}^N\int_{\Delta_i}k^{Y_\Delta}(z)^p\, d\sigma(z)
\lesssim
\sum_{i=1}^N
\sigma(\Delta_i)
\left(\fint_{2\,\Delta_i}k^{Y_\Delta}(z)\, d\sigma(z)\right)^p
\\[4pt]
\lesssim
\sum_{i=1}^N
\sigma(2\,\Delta_i)^{1-p}\,\omega^{Y_\Delta}(2\Delta_i)
\lesssim
\sigma(\hat{\Delta})^{1-p},
\end{multline*}
where in the last estimate we have used that $\omega(E)^{Y_\Delta}\lesssim 1$, the ADR property and and that $N$ is uniformly bounded. This gives $(b)$ in Theorem \ref{t1}, which in turn can be applied to obtain that $E$ is UR as desired. \qed

\parskip=0.1cm

\end{document}